\documentclass[11pt]{article}
\usepackage[margin=1in]{geometry}
\usepackage{microtype}
\usepackage{amsmath,amssymb,amsfonts,amsthm,mathtools,mathrsfs}
\usepackage{bm}
\usepackage{graphicx}
\usepackage{subcaption}   
\usepackage{float}        
\usepackage{placeins}     
\usepackage{longtable}
\usepackage{blkarray}
\usepackage{changepage}
\usepackage[shortlabels]{enumitem}
\usepackage[linesnumbered,lined,ruled]{algorithm2e}
\usepackage{tikz}
\usetikzlibrary{positioning,fit}
\usetikzlibrary{calc,positioning}
\usepackage{xcolor}
\usepackage{color}
\definecolor{blue}{rgb}{0,0,0.9} 
\definecolor{red}{rgb}{0.9,0,0} 
\definecolor{green}{rgb}{0,0.9,0}

\usepackage{comment}
\usepackage[
  colorlinks=true,
  linkcolor=blue,
  citecolor=blue,
  urlcolor=blue
]{hyperref}
\usepackage[nameinlink,noabbrev]{cleveref}
\usepackage[style=numeric, backend=biber]{biblatex}
\numberwithin{equation}{section}

\theoremstyle{plain}
\newtheorem{thm}{Theorem}[section]

\newtheorem{lem}[thm]{Lemma}

\newtheorem{prop}[thm]{Proposition}

\theoremstyle{definition}

\theoremstyle{remark}

\makeatletter
\newcommand{\myitem}[2][]{%
  \item[#1]%
  \protected@edef\@currentlabel{#1}%
  \label{#2}%
}
\makeatother

\newcommand{\dd}{\mathop{}\!\mathrm{d}}
\newcommand{\R}{\mathbb{R}}

\newcommand{\W}{\mathcal{W}}
\newcommand{\A}{\mathcal{A}}

\newcommand{\mF}{\mathrm{F}}
\newcommand{\C}{\mathcal{C}}
\newcommand{\B}{\mathcal{B}}
\newcommand{\N}{\mathcal{N}}
\newcommand{\E}{\mathcal{E}}
\newcommand{\G}{\mathcal{G}}

\newcommand{\cD}{\mathcal{D}}

\newcommand{\mom}{\mathrm{Mom}}
\newcommand{\ising}{\mathrm{Ising}}
\newcommand{\fit}{\mathrm{fit}}
\newcommand{\glb}{\mathrm{glb}}
\newcommand{\mpath}{\mathrm{path}}
\newcommand{\st}{\text{s.t.}\quad}
\newcommand{\plh}{\, \cdot \,}
\newcommand{\PL}{\paren{{\rm P}_{\rm L}}_\sharp}
\newcommand{\PR}{\paren{{\rm P}_{\rm R}}_\sharp}
\newcommand{\paren}[1]{\left( #1 \right)}
\newcommand{\bracket}[1]{\left[ #1 \right]}
\newcommand{\set}[1]{\left\{ #1 \right\}}

\newcommand{\angles}[1]{\left\langle #1 \right\rangle}

\let\svthefootnote\thefootnote
\newcommand{\blankfootnote}[1]{%
  \let\thefootnote\relax\footnotetext{#1}%
  \let\thefootnote\svthefootnote%
}

\title{Convex Relaxations for the Optimization of Markov Processes}

\author{Hongyi Zhang\thanks{Department of Statistics, University of Chicago, ({\tt hongyi518@uchicago.edu}).
         }, \quad
Yuehaw Khoo
	\thanks{Department of Statistics, University of Chicago, 
	({\tt ykhoo@uchicago.edu}). The research of this author is partially funded by NSF DMS-2339439, DOE DE-SC0022232, DARPA The Right Space HR0011-25-9-0031, and a Sloan research fellowship.}, \quad
	Tianyun Tang
\thanks{Department of Statistics, University of Chicago, ({\tt ttang@u.nus.edu}).
         }
	 }
	\date{\today}

\date{}

\begin{document}
\maketitle

\begin{abstract}
In this paper, we study the problem of optimizing Markov processes that interpolate between two prescribed probability distributions while minimizing a given cost. The main computational challenge is the curse of dimensionality: in high-dimensional state spaces, representing the full distribution is intractable. To address this, we reformulate the problem in terms of sequential couplings and develop convex relaxations based on local marginals and cluster moments. These relaxations exploit locality and sparse interaction structure, provide computable lower bounds, and recover low-order statistics of the intermediate laws. We identify dynamic optimal transport as a special case of our Markov process optimization problem and develop a procedure for recovering the underlying Benamou--Brenier dynamics from the relaxed solution. We also show that the procedure extends to more general Markov processes and illustrate it with a constrained process between Ising models.

\end{abstract}

\section{Introduction}

\subsection{Optimizing over Markov processes}\label{subsec1.1:problem}
In this paper, we study the problem of finding a Markov process that connects two prescribed distributions while minimizing a given cost. Let $(\Omega,\B(\Omega))$ be a Borel state space. We denote by $\mathcal{M}(\Omega)$ the set of finite signed Borel measures on $\Omega$, and by
\begin{align*}
    \mathcal{P}(\Omega):=\{\rho\in\mathcal{M}(\Omega):\rho\ge 0,\ \rho(\Omega)=1\}
\end{align*}
the set of probability measures on $\Omega$. Given initial and terminal distributions $\mu,\nu\in\mathcal{P}(\Omega)$, and one-step cost functions $c^s:\Omega\times\Omega\to\R$ for $s=0,\ldots,T-1$, we seek intermediate laws $\rho^s\in\mathcal{P}(\Omega)$ and Markov transition kernels $K^s(x,dy)$ solving
\begin{align}
    \inf_{\set{K^s}_{s=0}^{T-1}, \set{\rho^s}_{s=0}^T } &    \sum_{s=0}^{T-1} \int_\Omega\int_\Omega c^s(x,y)\,K^s(x,\dd y)\dd \rho^s(x) \label{problem:1.1}  \\
    \st & \rho^0 = \mu,\quad \rho^T = \nu, \label{constraint-endpoint 1.1} \tag{1.1-a}\\
    & \rho^{s+1} = \paren{K^s}^* \rho^s, \quad s=0,\ldots,T-1, \label{constraint-Markov 1.1} \tag{1.1-b}\\
    & K^s \in \mathcal K^s, 
    % \quad K^s(x,\cdot)\in\mathcal{P}(\Omega), 
    \quad s=0,\ldots,T-1, \label{constraint-control 1.1} \tag{1.1-c} \\
    & \rho^s \ge 0, \quad \rho^s(\Omega) = 1,\quad s=0,\ldots,T. \label{constraint-positivity 1.1} \tag{1.1-d}
\end{align}
Here $\mathcal{K}^s$ denotes the admissible class of Markov kernels at step $s$. The kernels are allowed to depend on $s$, so the process is generally time-inhomogeneous. In the Markov evolution constraint \eqref{constraint-Markov 1.1}, $(K^s)^*$ denotes the adjoint action of the kernel on probability measures: if $K^s$ acts on test functions by
\begin{align*}
    K^s f(x)=\int_{\Omega}f(y)\,K^s(x,dy),
\end{align*}
then, for every Borel set $B\in\mathcal{B}(\Omega)$ and probability measure $\rho\in\mathcal{P}(\Omega)$,
\begin{align*}
    (K^s)^*\rho(B)=\int_{\Omega}K^s(x,B)\dd \rho(x).
\end{align*}

A special case of problem \eqref{problem:1.1} is dynamic optimal transport. Classical optimal transport chooses a coupling between two endpoint distributions and minimizes a transportation cost \cite{kantorovich1942translocation,villani2009optimal}. Dynamic optimal transport instead seeks a curve of probability measures $\set{q(\plh, t)\in \mathcal{P}(\Omega): t\in[0,1]}$ with prescribed initial and terminal laws $q(\plh,0)$ and $q(\plh,1)$. When $\Omega \subset \R^d$, the Benamou--Brenier formula realizes this path through a velocity field \cite{benamou2000}:
\begin{align}\label{problem:BB}
    \W_p^p(\mu,\nu)=\inf_{q, v}\left\{\int_0^1\int_{\Omega}\|v(x,t)\|^p\dd q(x,t)\dd t:\partial_t q +\nabla \cdot(v\,q)=0, \; q(\plh, 0)=\mu, \;q(\plh, 1)=\nu\right\}. \tag{BB}
\end{align}
Here $\W_p$ denotes the $p$-Wasserstein distance and $\set{v(\plh, t): \Omega \rightarrow \R^d \mid t \in [0,1]}$ is the velocity field transporting the mass. In \Cref{sec:3}, we show that, after time discretization, the unconstrained kinetic-cost case of problem \eqref{problem:1.1} exactly recovers the optimal value and the grid-time measures of the Benamou--Brenier problem~\eqref{problem:BB}.

Beyond the unconstrained kinetic-cost case, Problem~\eqref{problem:1.1} also describes more general time-discrete distributional dynamics through the admissible kernel classes $\mathcal K^s$. Depending on the application, these classes may encode support, transition-probability, or local-update constraints on the one-step evolution. The single-spin process studied in \Cref{sec:Ising} is one concrete example. We focus on settings in which the corresponding coupling constraints admit tractable convex local outer approximations.

\subsection{Sequential-coupling formulation}
The Markov kernel formulation \eqref{problem:1.1} can equivalently be written as a sequential-coupling formulation. Given $\rho^s$ and $K^s$, define
\begin{align}\label{eq:pi & K}
    \pi^s(\dd x,\dd y) := \rho^s(\dd x)K^s(x,\dd y). 
\end{align}
Then $\pi^s\in\mathcal{P}(\Omega\times\Omega)$ is the joint law of two consecutive states $(X^s,X^{s+1})$. Let ${\rm P}_{\rm L},{\rm P}_{\rm R}:\Omega\times\Omega\to\Omega$ be the canonical projections onto the first and second components. The marginals of $\pi^s$ satisfy
\begin{align}\label{eq:rho & pi}
    \PL \pi^s = \rho^s, \quad \PR \pi^s = \rho^{s+1}, \quad s = 0, \ldots, T-1
\end{align}
where $\paren{\plh}_\sharp $ denotes the push-forward operator. Conversely, assuming $\Omega$ is a standard Borel space, any joint law $\pi^s$ admits a disintegration with respect to its first marginal, yielding a Markov transition kernel $K^s$ that satisfies \eqref{eq:pi & K} and is unique $\rho^s$-almost everywhere. In the finite-state case, this is simply $K^s(y|x)=\pi^s(x,y)/\rho^s(x)$ whenever $\rho^s(x)>0$.
% while $K^s(\cdot|x)$ may be chosen arbitrarily on states with $\rho^s(x)=0$. 
Therefore, optimizing over the transition kernels $K^s$ is equivalent to optimizing over the joint couplings $\pi^s$. Let $\cD^s$ be the sequential-coupling form of the Markov kernel constraints \eqref{constraint-control 1.1}: 
\begin{equation}\label{defiD}
    \cD^s:=\left\{ \pi^s\in \mathcal{P}(\Omega\times \Omega):\ \exists\ \rho^s\in \mathcal{P}(\Omega),\ K^s\in \mathcal{K}^s\ {\rm s.t.}\ \pi^s(\dd x,\dd y)=\rho^s (\dd x)K^s(x,\dd y) \right\},
\end{equation}
which means a nonnegative joint law $\pi^s$ belongs to $\cD^s$ if and only if its conditional transition kernel belongs to $\mathcal{K}^s$. Then \eqref{problem:1.1} can equivalently be written as
\begin{align}
    \inf_{\{\pi^s\}_{s=0}^{T-1}} \; &  \sum_{s=0}^{T-1} \int_{\Omega\times\Omega} c^s(x,y)\dd \pi^s(x,y) \label{problem:coupling}\\
    \st & \PL \pi^0 = \mu, \quad \PR \pi^{T-1} = \nu, \label{constraint-endpoint coupling} \tag{1.5-a}\\
    & \PR \pi^s = \PL \pi^{s+1}, \quad s=0,\ldots,T-2, \label{constraint-time consis coupling} \tag{1.5-b}\\
    & \pi^s\in \mathcal D^s, \quad s=0,\ldots,T-1, \label{constraint-control coupling} \tag{1.5-c}\\
    & \pi^s \ge 0, \quad \pi^s(\Omega \times \Omega)=1,\quad s=0,\ldots,T-1.\label{constraint-positivity coupling} \tag{1.5-d}
\end{align}
The intermediate distributions $\set{\rho^s}_{s=0}^{T}$ are absorbed into the couplings: once $\{\pi^s\}_{s=0}^{T-1}$ is known, the intermediate marginals at the prescribed grid times are recovered as
\begin{align}
    \rho^0=\PL \pi^0, \quad \rho^{s+1}=\PR \pi^s, \quad s=0,\ldots,T-1.
\end{align}

Our central goal is to solve problem \eqref{problem:coupling} in high dimensions. This is challenging because representing the full couplings $\set{\pi^s}_{s=0}^{T-1}$ suffers from the curse of dimensionality \cite{fournier2015rate,weed2019sharp}. To alleviate this issue, we develop convex relaxation methods that retain only low-order local information of $\pi^s$, such as a sparse collection of marginals or cluster moments. These relaxations, developed in \Cref{sec:conv relaxation}, provide computable lower bounds for problem \eqref{problem:coupling} and recover low-order statistics of the intermediate distributions. 

It remains to recover Markov kernels in the original problem \eqref{problem:1.1}. If the full couplings were available, the kernels could be obtained by disintegration as in \eqref{eq:pi & K}. After relaxation, however, only partial local information is available. In the dynamic optimal transport case, the Benamou--Brenier dynamics is represented by a velocity field; \Cref{sec:3} shows how this velocity can be recovered from dual variables of \eqref{problem:coupling}, yielding an associated Markov kernel through the induced flow. These dual variables can, in turn, be approximated by solutions of the dual convex relaxation. For more general constrained Markov processes, we fit a parametrized family of kernels to the recovered local statistics. \Cref{sec:Ising} develops this fitting procedure and illustrates it with a Markov process between Ising distributions based on Glauber dynamics.

\subsection{Our contributions}
We summarize the main contributions.

\begin{itemize}
    \item We recast the Markov process optimization problem \eqref{problem:1.1} into a sequential-coupling formulation \eqref{problem:coupling}, and propose a convex relaxation framework for solving it in high dimensions without incurring exponential computational costs. The framework can handle both discrete and continuous state spaces, and the relaxed solution provides low-order statistics of the intermediate distributions.

    \item We demonstrate that problem \eqref{problem:1.1} encompasses dynamic optimal transport as a special case, and we develop a procedure to recover the underlying Benamou--Brenier dynamics from the relaxed solution of \eqref{problem:coupling}.
    
    \item We show that our procedure extends beyond dynamic optimal transport to handle other types of dynamics. As a concrete example, we illustrate our method with a Markov process between Ising models governed by Glauber dynamics.
\end{itemize}

\subsection{Related work}

Dynamic optimal transport is a central special case of our framework. The Kantorovich problem and the Benamou--Brenier formula give the classical static and dynamic viewpoints on Wasserstein transport \cite{kantorovich1942translocation, villani2009optimal, benamou2000}. After the momentum substitution, the Benamou--Brenier problem is convex. Eulerian augmented-Lagrangian and proximal-splitting methods exploit this structure effectively in low dimensions, but their space--time grids suffer from the curse of dimensionality \cite{benamou2000,Papadakis_2014proximalSplitting}. High-dimensional alternatives use Lagrangian or sample-based representations and neural parameterizations of velocity fields or flow maps, including neural discretizations of velocity fields and flow maps \cite{wan2022scalabledeeplearningapproach,finlay2020neuralOED,xu2025Qflow}, TrajectoryNet-type methods \cite{tong2020trajectoryNet}, and flow-matching variants based on minibatch optimal transport or convex parameterizations \cite{tong2024OT-CFM,kornilov2024OFM}. These methods avoid full spatial grids and scale well empirically, but their training problems are generally nonconvex. In finite-sample implementations, endpoint matching and kinetic costs are estimated from samples, minibatch couplings, or penalties; sampling, approximation, optimization, and ODE-discretization errors are therefore intertwined, intermediate laws are usually assessed empirically, and certified lower bounds for the original unregularized action are generally unavailable. Our approach instead optimizes local marginals or cluster moments through a convex program, providing a computable lower bound, low-order intermediate statistics, and approximate dual velocity information. Before spatial relaxation, the sequential-coupling formulation also has the exact Benamou--Brenier value on any prescribed time grid.

Beyond the unconstrained Benamou--Brenier setting, a related literature restricts the admissible distributional dynamics themselves. Constrained dynamic optimal transport on parameterized families and dynamical transport for nonlinear control-affine systems replace the free velocity field by model-specific evolution laws \cite{li2018constrainedDOT,Karthik2023nonlinearDOT}. After time discretization and, when possible, elimination of auxiliary controls, such models may induce one-step kernel or coupling constraints of the type considered here. Their numerical treatment, however, is generally tailored to the prescribed dynamics; in particular, the control-affine method in \cite{Karthik2023nonlinearDOT} still relies on spatial discretization.

Schr\"odinger bridge problems also connect prescribed endpoint distributions through Markovian dynamics, but minimize path-space relative entropy with respect to a prior process \cite{leonard2014survey,chen2016SBrelationOT}. Recent diffusion and discrete variants provide scalable stochastic interpolations \cite{DeBortoliThorntonHengDoucet2021, ShiDeBortoliCampbellDoucet2023,KimEtAl2025}. By contrast, our formulation allows general additive one-step costs and hard constraints on admissible transition kernels, without requiring a reference process or entropic regularization. 

At the relaxation level, moment--SOS methods provide convex hierarchies for generalized moment and polynomial optimization \cite{lasserre2009moments, parrilo2020sum}, and have been applied to static optimal transport \cite{mula2022momentsos}. The associated dense moment matrices, however, grow rapidly with dimension. Sparse SOS hierarchies address this issue by exploiting correlative sparsity \cite{WakiKimKojimaMuramatsu2006,Lasserre2006}; more directly, the cluster marginal and moment relaxations in \cite{khoo2025} provide the closest methodological antecedent to our approach for high-dimensional static transport. Convex moment relaxations have also been developed for Fokker--Planck dynamics \cite{covFokkerPlack}. Our framework combines these ideas by applying local marginal or moment representations to each adjacent coupling and linking them through multiple constraints. The resulting program is therefore a convex relaxation of the full multistage problem, rather than a collection of independent static relaxations.

\subsection{Organization}
The rest of the paper is organized as follows. \Cref{sec:conv relaxation} develops the marginal and moment relaxations for the sequential-coupling problem \eqref{problem:coupling}. \Cref{sec:3} identifies dynamic optimal transport as a special case of problem \eqref{problem:coupling} and shows how the underlying Benamou--Brenier dynamics can be recovered from the relaxed solution of \eqref{problem:coupling}. \Cref{sec:Ising} develops a kernel-fitting procedure for more general Markov processes and illustrates it with a Markov process between Ising models governed by Glauber dynamics. \Cref{sec:experiments} presents numerical experiments and \Cref{sec:conclusion} concludes the paper.

\section{Convex relaxation} \label{sec:conv relaxation}

In this section, we develop a convex relaxation framework for problem \eqref{problem:coupling}. The construction is motivated by the marginal and cluster moment relaxations for static optimal transport in~\cite{khoo2025}, whose basic idea is to exploit the locality and sparse interaction structure of the problem by representing a coupling $\pi$ only through its local information on selected coordinate clusters and cluster pairs. In the sequential-coupling formulation, this local representation is applied to every coupling $\pi^s$, and the constraints in \eqref{problem:coupling} are imposed through the retained local variables. In \Cref{subsec:2.1 notation}, we introduce the basic notation and cluster structure used for this representation. \Cref{subsec:marginal relaxation} presents the marginal relaxation, which is natural for finite or discretized state spaces. \Cref{subsec:moment relaxation} presents the moment relaxation, which is well suited to continuous state spaces, where direct marginal discretization is computationally costly.

\subsection{Notation and cluster graph}\label{subsec:2.1 notation}
Let $[d]=\{1,\ldots,d\}$ be the coordinate set, and assume that the state space has the product form $\Omega=\Omega_1\times\cdots\times\Omega_d$. For $x\in\Omega$, write $x=(x_1,\ldots,x_d)$ where $x_i\in\Omega_i$. A cluster is a subset of coordinates. We choose a collection of clusters
\begin{align*}
    \C = \set{A_1, \ldots, A_K}, \quad K = |\C|, \; A_i \subset \bracket{d}, \; 1 \le i \le K
\end{align*}
to specify which local groups of coordinates will be represented explicitly. The purpose of the clusters is to retain local marginals while avoiding a full representation on all $d$ coordinates. If each cluster is small, then distributions on one cluster or on a pair of clusters remain low-dimensional while still capturing local dependence structure. Unless otherwise stated, we assume that $\C$ is a partition of $\bracket{d}$:
\begin{align*}
    A_a\cap A_b=\varnothing \quad (a\ne b), \quad \bigcup_{a=1}^K A_a=[d].
\end{align*}
For any $A\subset[d]$, define
\begin{align}
    x_A=(x_i)_{i\in A}, \quad \Omega_A=\prod_{i\in A}\Omega_i.
\end{align}
For an adjacent pair $z = (x,y)\in\Omega\times\Omega$, define
\begin{align}
    z_A=(x_A,y_A), \quad z_A \in Z_A=\Omega_A\times\Omega_A.
\end{align}
Thus $z_A$ contains the coordinates in cluster $A$ at two consecutive time layers. Throughout the paper, lowercase subscripts such as $i,j$, together with numerical subscripts such as $1,2$, refer to individual coordinates, whereas capital subscripts such as $A,B$ refer to clusters.

The pairwise information between clusters retained by the relaxation is encoded by a cluster graph 
\begin{align}
    \G=(\C,\E),
\end{align}
where the vertex set $\C$ is the family of all clusters and $\E$ is the edge set:
\begin{align*}
    \E \subset \bracket{\C}_2, \quad  \bracket{\C}_2 = \set{AB:A,B\in\C, A \neq B}.
\end{align*}
For brevity, we write $AB$ for the unordered pair $\set{A,B}$. An edge $AB\in\E$ means that we retain the joint information between clusters $A$ and $B$. Sparse choices of $\G$ lead to smaller convex programs. 

For a probability measure $\eta$ on $\Omega$, its one-cluster and cluster-pair marginals are denoted by
\begin{align}
    \eta_A=({\rm P}_A)_{\sharp} \, \eta, \quad \eta_{AB}=({\rm P}_{AB})_{\sharp} \, \eta, \qquad A \in \C, \quad AB \in [\C]_2,
\end{align}
where ${\rm P}_A:\Omega\to\Omega_A$ and ${\rm P}_{AB}:\Omega\to\Omega_A\times\Omega_B$ are the canonical projections. Similarly, for a coupling $\pi^s$ on $\Omega\times\Omega$, we denote its one-cluster and cluster-pair marginals by
\begin{align}
    \pi^s_A = ({\rm P}_A )_\sharp \, \pi^s, \quad \pi^s_{AB} = ({\rm P}_{AB})_\sharp \, \pi^s, \qquad A \in \C, \quad AB \in [\C]_2.
\end{align}
Here, by a slight abuse of notation, the same symbols $\mathrm{P}_A$ and $\mathrm{P}_{AB}$ are used for the corresponding projections on $\Omega \times \Omega$.

For an array of functions $\set{f_a(z)}_{a \in \A}$ over the two-layer variable $z = (x,y) \in \Omega \times \Omega$, we introduce selection operators ${\rm R}_x$ and ${\rm R}_y$, which retain only the entries that depend on the left and right time layers, respectively.
\begin{align}\label{def:RxRy}
    \begin{split}
        {\rm R}_x(\set{f_a(z)}_{a \in \mathcal{A}}) \coloneqq \set{f_a \mid \exists \, \widetilde{f} \text{ such that } f_a(x,y) = \widetilde{f}(x), a \in \A }, \\
        {\rm R}_y(\set{f_a(z)}_{a \in \A}) \coloneqq \set{f_a \mid \exists\, \widetilde{f} \text{ such that } f_a(x,y) = \widetilde{f}(y), a \in \A}.
    \end{split}
\end{align}
For example, applying ${\rm R}_x$ and ${\rm R}_y$ to a matrix of basis functions yields
\begin{align*}
    {\rm R}_x\left(
    \begin{bmatrix}
        1 & y_2  \\ 
        x_1^2 & x_1y_2
    \end{bmatrix}\right) =
    \begin{bmatrix}
        1 \\ x_1^2
    \end{bmatrix}, \quad 
    {\rm R}_y\left(
    \begin{bmatrix}
        1 & y_2  \\ 
        x_1^2 & x_1y_2
    \end{bmatrix}\right) =
    \begin{bmatrix}
        1 \\ y_2
    \end{bmatrix}.
\end{align*}

Finally, given a finite signed measure $\eta$ on a measurable space $W$ and an integrable scalar-, vector-, or matrix-valued test function $\Xi:W\to\mathbb{C}^{m\times n}$, we write its moment under $\eta$ as
\begin{align}
    \eta(\Xi)=\int_W \Xi(w)\dd \eta(w), 
\end{align}
with the integral understood entrywise. For example, when $\Xi_{AB}$ is a function on the set $Z_A \times Z_B = \paren{\Omega_A \times \Omega_A} \times \paren{\Omega_B \times \Omega_B}$, its moment under the cluster-pair marginal $\pi_{AB}$ is 
\begin{align*}
    \pi_{AB}^s(\Xi_{AB})=\int_{Z_A\times Z_B}\Xi_{AB}(z_A,z_B)\dd \pi_{AB}^s(z_A,z_B).
\end{align*}

\subsection{Marginal relaxation}\label{subsec:marginal relaxation}
In this subsection, we present the marginal relaxation of problem \eqref{problem:coupling}. Given a specified cluster decomposition $\C$ of the coordinate set $[d]$, which is intended to group strongly correlated coordinates together, the full coupling $\pi^s(x,y)$ on each time interval $s = 0, \ldots, T-1$ is replaced by its local marginals on clusters and cluster pairs
\begin{align}\label{notation:reduced var Mar}
    \{\pi_A^s\}_{A\in\C}, \quad \{\pi_{AB}^s\}_{AB\in \bracket{\C}_2}.
\end{align}
Here $\pi_A^s$ is a probability measure on $Z_A$, and $\pi_{AB}^s$ is a probability measure on $Z_{A}\times Z_B$. In practice, the size of each cluster is usually much smaller than $d$, so these local objects can be represented directly on the chosen finite state space or spatial grid. We also use the cluster graph $\G = (\C, \E)$ introduced above. Its edge set specifies which cluster-pair marginals are retained in the objective and subsequent constraints.
 
Assume that the one-step cost either admits the local decomposition
\begin{align}\label{eq:relaxed cost Mar}
    c^s(z) = \sum_{A\in\C} c_{A}^s(z_A) + \sum_{AB\in\E} c_{AB}^s(z_A,z_B),
\end{align}
or is approximated by the right-hand side. If the decomposition is exact and the local marginals are induced by a full coupling $\pi^s$, then
\begin{align}
    \int_{\Omega\times\Omega}c^s(z)\dd \pi^s(z) = \sum_{A\in\C}\int_{Z_A}c_A^s(z_A)\dd \pi_A^s(z_A)+\sum_{AB\in\E}\int_{Z_A\times Z_B}c_{AB}^s(z_A,z_B)\dd \pi_{AB}^s(z_A,z_B). 
\end{align}
For example, the quadratic cost $c^s(x,y)= \|x-y\|^2$ decomposes over a partition $\C$ as $\|x-y\|^2 = \sum_{A\in\C} \|x_A-y_A\|^2$. Thus, in this case, only one-cluster cost terms are needed. More general local costs may include the cluster-pair terms $c_{AB}^s$.

We next impose constraints on the reduced variables.
\paragraph{1. Marginal consistency constraints.}
For every pair $AB \in [\C]_2$ and each $s = 0, \ldots, T-1$, the one-cluster and cluster-pair marginals must be compatible in the sense that
\begin{equation}\label{constraint-local consis Mar}
    ({\rm P}_{A} )_\sharp \, \pi^s_{AB} = \pi^s_A, \quad ({\rm P}_B)_\sharp  \, \pi^s_{AB} = \pi^s_B, \quad \forall AB \in [\C]_2.
\end{equation}
These constraints are necessary if the family $\set{\pi_A^s, \pi_{AB}^s}$ is to be interpreted as the collection of one-cluster and cluster-pair marginals of a common global coupling $\pi^s$.

\paragraph{2. Initial and terminal constraints.}
The initial and terminal constraints in the full sequential-coupling problem, namely \eqref{constraint-endpoint coupling}, require the left marginal of the first coupling to be $\mu$ and the right marginal of the last coupling to be $\nu$. In the marginal relaxation, these constraints are relaxed to hold only on the retained cluster-pair marginals:
\begin{align}\label{constraint-endpoint Mar}
    \PL \pi_{AB}^0=\mu_{AB}, \quad \PR \pi_{AB}^{T-1}=\nu_{AB}, \quad AB\in\E.
\end{align}
If every cluster is incident to at least one edge in $\E$, then the corresponding initial and terminal constraints for one-cluster marginals follow from \eqref{constraint-local consis Mar} and \eqref{constraint-endpoint Mar}. If isolated clusters are allowed, their initial and terminal constraints are added explicitly.

\paragraph{3. Mass conservation constraints.}
The right marginal of $\pi^s$ can be interpreted as the distribution of mass arriving at time $t_{s+1}$, whereas the left marginal of $\pi^{s+1}$ represents the distribution of mass leaving at time $t_{s+1}$. The mass conservation constraints \eqref{constraint-time consis coupling} ensure that these two marginals coincide. In the marginal relaxation, this condition is imposed after projection onto the retained cluster pairs:
\begin{align}\label{constraint-time consis Mar}
    \PR \pi_{AB}^s=\PL \pi_{AB}^{s+1}, \quad AB\in\E, \quad s = 0,\ldots T-2. 
\end{align}
Thus \eqref{constraint-time consis Mar} is the projected local form of mass conservation across adjacent time intervals.

\paragraph{4. Markov kernel constraints.}
The admissible-kernel constraint in the Markov formulation, $K^s \in \mathcal{K}^s$, is expressed in the sequential-coupling formulation as $\pi^s \in \cD^s$ where $\cD^s$ denotes the admissible class of couplings at step $s$. At the cluster level, we replace $\cD^s$ by a convex local admissible set $\cD_{\mathrm{Mar}}^s$, and we impose
\begin{align}\label{constraint-control Mar}
    \{\pi_A^s,\pi_{AB}^s\}_{A\in\C,\,AB\in[\C]_2}\in\cD_{\mathrm{Mar}}^s,\qquad s=0,\ldots,T-1. 
\end{align}

The set $\cD_{\mathrm{Mar}}^s$ is a convex outer approximation that contains the local marginals induced by any admissible coupling $\pi^s \in \mathcal{D}^s$. Thus, this constraint remains a relaxation of the original Markov kernel constraint.

\paragraph{5. Local positivity constraints.}
The nonnegativity constraint \eqref{constraint-positivity coupling} for the full coupling implies nonnegativity of all its local marginals. We impose this at the cluster-pair level:
\begin{align}\label{constraint-local positivity Mar}
    \pi_{AB}^s \ge 0, \quad AB \in [\C]_2,\quad s = 0,\cdots T-1.
\end{align}
The one-cluster nonnegativity $\pi_A^s \ge 0$ follows from \eqref{constraint-local consis Mar} and \eqref{constraint-local positivity Mar}, provided each cluster appears in at least one retained pair. If isolated one-cluster variables are allowed, their nonnegativity is imposed explicitly.

\paragraph{6. Global positivity constraints.}
The preceding local constraints do not guarantee that the family $\set{\pi_A^s, \pi_{AB}^s}_{A,AB}$ is induced by a genuine global coupling $\pi^s$. We therefore impose the following necessary global positivity condition. If the local variables are induced by a genuine coupling $\pi^s$, then for every family of square-integrable test functions $\set{f_A \in L^2(\pi_A^s)}_{A \in \C} $, we have
\begin{align*}
    & \sum_{A\in\C}\int_{Z_A} f_A(z_A)^2 \dd \pi_A^s(z_A) + 2\sum_{AB\in[\C]_2}\int_{Z_A\times Z_B} f_A(z_A)f_B(z_B) \dd \pi_{AB}^s(z_A,z_B) \nonumber \\
    & = \pi^s\paren{\paren{\sum_{A \in \C}f_A(z_A)}^2} \ge 0. 
\end{align*}
In the finite-state or gridded setting, this condition is equivalent to the PSD constraint
\begin{align}\label{constraint-PSD Mar}
    \begin{bmatrix}
        \operatorname{Diag}(\pi_{A_1}^s) & \pi_{A_1A_2}^s & \cdots & \pi_{A_1A_K}^s\\
        (\pi_{A_1A_2}^{s})^\top & \operatorname{Diag}(\pi_{A_2}^s) & \cdots & \pi_{A_2A_K}^s\\ 
        \vdots & \vdots & \ddots & \vdots\\ 
        (\pi_{A_1A_K}^s)^{\top} & (\pi_{A_2A_K}^s)^{\top} & \cdots & \operatorname{Diag}(\pi_{A_K}^s)
    \end{bmatrix}
    \succeq 0, \quad s=0,\ldots, T-1,
\end{align}
where $K = |\C|$ is the number of clusters and $\operatorname{Diag}(\pi_{A_k}^s)$ denotes the diagonal matrix with diagonal entries given by $\pi_{A_k}^s$. We write this compactly as $(\pi_A^s,\pi_{AB}^s)_{A \in \C, AB \in [\C]_2}\succeq 0$.

Combining the preceding constraints gives the full marginal relaxation:
\begin{align}
    \min_{\set{\pi^s_A}_{\C}, \set{\pi^s_{AB}}_{[\C]_2}} \; & \sum_{s=0}^{T-1} \bracket{\sum_{A\in\C}\int_{Z_A}c_A^s\dd \pi_A^s+\sum_{AB\in\E}\int_{Z_A\times Z_B}c_{AB}^s\dd \pi_{AB}^s}
    \label{problem:Mar} \tag{Mar}\\
    \st \quad &  \set{\pi_A^s, \pi_B^s, \pi_{AB}^s}_{AB \in [\C]_2} \text{ is consistent \eqref{constraint-local consis Mar}}, \label{constraint-local consis Mar-copy}\tag{Mar-a}\\
    & \text{initial and terminal constraints \eqref{constraint-endpoint Mar} for } \set{\pi_{AB}^i}_{\E}, \; i=0 \text{ or } T-1,
    \label{constraint-endpoint Mar-copy}\tag{Mar-b}\\
    & \text{mass conservation \eqref{constraint-time consis Mar} for } \set{\pi_{AB}^s}_{\E}, 
    \label{constraint-time consis Mar-copy} \tag{Mar-c}\\
    & \text{Markov kernel constraints \eqref{constraint-control Mar} for } \set{\pi_A^s, \pi_{AB}^s}_{\C, [\C]_2}, 
    \label{constraint-control Mar-copy} \tag{Mar-d}\\
    & \text{local positivity \eqref{constraint-local positivity Mar} for } \set{\pi_{AB}^s}_{[\C]_2},  
    \label{constraint-local positivity Mar-copy} \tag{Mar-e} \\
    & \text{global positivity \eqref{constraint-PSD Mar}: } (\pi_{A}^s,\pi_{AB}^s)_{\C, \, [\C]_2}\succeq 0. 
    \label{constraint-PSD Mar-copy} \tag{Mar-f}
\end{align}
If $\cD^s_{\mathrm{Mar}}$ is described by affine equalities and inequalities, then \eqref{problem:Mar} is a doubly nonnegative block semidefinite program (SDP): it combines pairwise entrywise nonnegativity with a PSD constraint for each time interval. 

Following the approach in \cite{khoo2025}, we also consider two simpler relaxations of \eqref{problem:Mar}. The first simplification drops the global positivity constraints \eqref{constraint-PSD Mar-copy} and retains pair variables only on the edge set $\E$. Accordingly, the marginal consistency, Markov kernel, and local positivity constraints are imposed only for $AB \in \E$. The non-edge variables $\pi_{AB}^s$ with $AB\notin\E$ are then omitted. This gives
\begin{align}
    \min_{\set{\pi^s_A}_{\C}, \set{\pi^s_{AB}}_{\E}} \; & \sum_{s=0}^{T-1}\bracket{\sum_{A\in\C}\int_{Z_A}c_A^s\dd \pi_A^s+\sum_{AB\in\E}\int_{Z_A\times Z_B}c_{AB}^s\dd \pi_{AB}^s} 
    \label{problem:Mar1} \tag{$\mathrm{Mar}^1$}\\
    \st \quad & \set{\pi_A^s, \pi_B^s, \pi_{AB}^s}_{AB \in \E} \text{ is consistent \eqref{constraint-local consis Mar}},  
    \label{constraint-endpoint Mar1}\tag{$\mathrm{Mar}^1$-a}\\
    & \text{initial and terminal constraints \eqref{constraint-endpoint Mar} for } \set{\pi_{AB}^i}_{\E}, \; i = 0 \text{ or } T-1,
    \label{constraint-local consis Mar1}\tag{$\mathrm{Mar}^1$-b}\\
    & \text{mass conservation \eqref{constraint-time consis Mar} for } \set{\pi_{AB}^s}_{\E},
    \label{constraint-time consis Mar1} \tag{$\mathrm{Mar}^1$-c}\\
    & \text{projected Markov kernel constraints: } \set{\pi_A^s, \pi_{AB}^s}_{\C, \, \E} \in \cD^s_{\mathrm{Mar}, \E},
    \label{constraint-control Mar1} \tag{$\mathrm{Mar}^1$-d}\\
    & \text{local positivity \eqref{constraint-local positivity Mar} for } \set{\pi_{AB}^s}_{\E}, 
    \label{constraint-local positivity Mar1} \tag{$\mathrm{Mar}^1$-e} 
\end{align}
Here $\cD^s_{\mathrm{Mar}, \E}$ in \eqref{constraint-control Mar1} denotes the projection of the constraint set $\cD^s_{\mathrm{Mar}}$ onto the retained variables. When $\Omega$ is finite and these constraints are affine, \eqref{problem:Mar1} reduces to a linear program.

The second simplification drops the local positivity constraints \eqref{constraint-local positivity Mar-copy} while retaining the global positivity constraints. The resulting relaxation is
\begin{align}
    \min_{\set{\pi^s_A}_{\C}, \set{\pi^s_{AB}}_{[\C]_2}} \; & \sum_{s=0}^{T-1} \bracket{\sum_{A\in\C}\int_{Z_A}c_A^s\dd \pi_A^s+\sum_{AB\in\E}\int_{Z_A\times Z_B}c_{AB}^s\dd \pi_{AB}^s}
    \label{problem:Mar2} \tag{$\mathrm{Mar}^2$}\\
    \st \quad & \set{\pi_A^s, \pi_B^s, \pi_{AB}^s}_{AB \in [\C]_2} \text{ is consistent \eqref{constraint-local consis Mar}}, \label{constraint-local consis Mar2}\tag{$\mathrm{Mar}^2$-a}\\
    & \text{initial and terminal constraints \eqref{constraint-endpoint Mar} for } \set{\pi_{AB}^i}_{ \E}, i = 0 \text{ or } T-1,
    \label{constraint-endpoint Mar2}\tag{$\mathrm{Mar}^2$-b}\\
    & \text{mass conservation \eqref{constraint-time consis Mar} for } \set{\pi_{AB}^s}_{ \E}, 
    \label{constraint-time consis Mar2} \tag{$\mathrm{Mar}^2$-c}\\
    & \text{Markov kernel constraints \eqref{constraint-control Mar} for } \set{\pi_A^s, \pi_{AB}^s}_{\C,\, [\C]_2}, 
    \label{constraint-control Mar2} \tag{$\mathrm{Mar}^2$-d}\\
    & \text{global positivity \eqref{constraint-PSD Mar}: } (\pi_{A}^s,\pi_{AB}^s)_{\C, \,[\C]_2}\succeq 0. 
    \label{constraint-PSD Mar2} \tag{$\mathrm{Mar}^2$-e}
\end{align}
The variables $\pi_{AB}^s$ are now allowed to be signed measures. The global positivity constraints still imply $\pi_A^s\ge 0$, because each diagonal block in the finite-dimensional PSD matrix is $\operatorname{Diag}(\pi_A^s)\succeq 0$. However, they do not imply pointwise nonnegativity of the pair variables $\pi_{AB}^s$.

\subsection{Cluster moment relaxation}\label{subsec:moment relaxation}
We now introduce the cluster moment relaxation. As in the marginal relaxation, we fix a cluster decomposition $\C$ and the corresponding cluster graph $\G = (\C, \E)$. The moment relaxation replaces the local marginals used in the marginal relaxation by finitely many moments against prescribed cluster bases. We first define the cluster bases.

Fix a degree $r \in \mathbb{N}^+$. For each coordinate $i \in [d]$, choose a finite one-dimensional basis of order at most $r$, denoted by $\set{\phi_{i,j}: \Omega_i \rightarrow \mathbb{C}}_{j=0}^r$, where $\phi_{i,0} \equiv 1$. Typical choices include monomial bases, orthogonal polynomial bases, Fourier bases, or other bases adapted to the one-dimensional domain $\Omega_i$. Recall from \Cref{subsec:2.1 notation} that, for a cluster $A \in \C$, $z_A = (x_A, y_A) \in \Omega_A \times \Omega_A$, where $x_A$ and $y_A$ denote the variables of the left time layer $t_s$ and right time layer $t_{s+1}$, respectively. For a multi-index 
\begin{align*}
    \alpha = (\alpha^x, \alpha^y), \quad \alpha^x = (\alpha_i^x)_{i \in A}, \quad \alpha^y = (\alpha_i^y)_{i \in A}
\end{align*}
we define its total degree by
\[
    |\alpha| = \sum_{i \in A} \alpha_i^x + \sum_{i\in A} \alpha_i^y.
\]
Let $\mathcal{I}_A^r = \set{\alpha = (\alpha^x, \alpha^y) \mid |\alpha| \le r}$. For each $\alpha \in \mathcal{I}_A^r$, define the cluster basis function on $Z_A$ by
\begin{align*}
    \Phi_{A,\alpha}(z_A) = \prod_{i\in A}\phi_{i, \alpha_i^x}(x_i) \phi_{i, \alpha_i^y}(y_i).
\end{align*}
We collect these functions into a column vector $\Phi_A = \paren{\Phi_{A, \alpha}}_{\alpha \in \mathcal{I}_A^r}$, which defines the two-layer cluster basis associated with cluster $A$. This basis consists of functions supported only on the variables $z_A = (x_A, y_A)$ and having total degree at most $r$. For example, if $A = \set{1}$, $r = 2$ and $\phi_{1,j}(u_1) = u_1^j$, then, up to ordering, $\Phi_A = \paren{1, x_1, x_1^2, y_1, y_1^2, x_1y_1}^\top$.

We now define the moment variables. For each time interval $s = 0, \ldots, T-1$ and each cluster $A \in \C$, introduce the one-cluster moment matrix:
\begin{align}\label{eq:MA}
    M_A^s \coloneqq \pi_A^s(\Phi_A \Phi_A^*) = \int_{Z_A}\Phi_A(z_A)\Phi_A(z_A)^* \dd \pi_A^s(z_A),
\end{align}
where the integral is understood entrywise and $\paren{\cdot}^*$ denotes the Hermitian transpose. Similarly, for each cluster pair $AB \in [\C]_2$, the cross-moment matrix is defined as
\begin{align}\label{eq:MAB}
    M_{AB}^s \coloneqq \pi_{AB}^s(\Phi_A \Phi_B^*) = \int_{Z_A \times Z_B}\Phi_A(z_A)\Phi_B(z_B)^* \dd \pi_{AB}^s(z_A,z_B).
\end{align}
The decision variables in the moment relaxation are the collections
\begin{align}\label{notation:reduced var Mom}
    \{M_A^s\}_{A\in\C}, \quad \{M_{AB}^s\}_{AB\in \bracket{\C}_2}, \quad s = 0, \ldots, T-1.
\end{align}
These variables at each time step $s$ can be assembled into a single Hermitian matrix
\begin{align}\label{def:Ms}
    M^s:=
    \begin{bmatrix}
        M_{A_1}^s & M_{A_1A_2}^s & \cdots & M_{A_1A_K}^s \\ 
        \paren{M_{A_1A_2}^s}^* & M_{A_2}^s & \cdots & M_{A_2A_K}^s\\ 
        \vdots & \vdots & \ddots & \vdots\\ 
        \paren{M_{A_1A_K}^s}^* & \paren{M_{A_2A_K}^s}^* & \cdots & M_{A_K}^s
    \end{bmatrix}.
\end{align}
The overall decision variable is the block-diagonal matrix $M = \mathrm{Diag} \paren{M^0, \cdots, M^{T-1}}$.

We now describe the objective and constraints in terms of these moment variables. Recall from the previous subsection that we represent the one-step cost $c^s$ via a local decomposition \eqref{eq:relaxed cost Mar}:
\begin{align*}
    c^s(z) = \sum_{A\in\C} c_{A}^s(z_A) + \sum_{AB\in\E} c_{AB}^s(z_A,z_B).
\end{align*}
Assume that each local cost component lies in the span of the corresponding product basis, or is replaced by its projection onto that span. Let $C_A^s$ and $C_{AB}^s$ denote the representations of $c_A^s$ and $c_{AB}^s$ in the bases formed by the elements of $\Phi_A \Phi_A^*$ and $\Phi_A \Phi_B^*$, respectively. Specifically, 
\begin{align*}
    c_A^s=\langle C_A^s,\Phi_A\Phi_A^*\rangle, \quad c_{AB}^s =\langle C_{AB}^s,\Phi_A\Phi_B^*\rangle.
\end{align*}
In the exact case, the local objective at step $s$ can be written as
\begin{align} \label{eq:cost decomp moment}
    \int_{\Omega \times \Omega}c^s(x,y) \dd \pi^s(x,y) = \sum_{A\in\C}\langle C_{A}^s,M_A^s\rangle+\sum_{AB\in\E}\langle C_{AB}^s,M_{AB}^s\rangle.
\end{align}
We next impose constraints on the reduced moment variables. By the definition of selection operators in \eqref{def:RxRy}, ${\rm R}_x(\Phi_A\Phi_B^*)$ and ${\rm R}_y(\Phi_A\Phi_B^*)$ extract the submatrices of entries in $\Phi_A\Phi_B^*$ that depend exclusively on the left-layer variables $(x_A, x_B)$ and right-layer variables $(y_A,y_B)$, respectively. For each time interval $s = 0, \ldots, T-1$, cluster $A \in \C$ and cluster pair $AB \in [\C]_2$, we define the following shorthand notation:
\begin{equation}
    \begin{split}
        (M_A^s)_{\rm L} \coloneqq \pi_A^s({\rm R}_x(\Phi_A\Phi_A^*)), & \quad (M_A^s)_{\rm R} \coloneqq \pi_A^s({\rm R}_y(\Phi_A\Phi_A^*)), \\[1ex]
        (M_{AB}^s)_{\rm L} \coloneqq \pi_{AB}^s({\rm R}_x(\Phi_A\Phi_B^*)), & \quad (M_{AB}^s)_{\rm R} \coloneqq \pi_{AB}^s({\rm R}_y(\Phi_A\Phi_B^*)).
    \end{split}
\end{equation}

\paragraph{1. Moment consistency constraints.}
For each cluster $A\in\C$, the one-cluster moment matrix $M_A^s=\pi_A^s(\Phi_A\Phi_A^*)$ contains entries indexed by products of basis functions in the same variables $z_A=(x_A,y_A)$. Different pairs of basis functions may therefore generate the same product function. For example, if $A=\{1\}$ and the one-dimensional basis is the monomial basis $\phi_{1,j}(u_1)=u_1^j$, then the products $x_1\cdot x_1^3$ and $x_1^2\cdot x_1^2$ both represent the same function $x_1^4$. The corresponding entries of $M_A^s$ must therefore be equal. These identities give linear equality constraints on $M_A^s$, for every $A\in\C$ and $s=0,\ldots,T-1$. We denote them compactly by saying that 
\begin{align}\label{constraint-consis Mom}
    M^s \text{ is consistent, } \quad s = 0, \ldots, T-1.
\end{align}
Note that no analogous constraint is imposed on the cross-moment matrix $M_{AB}^s=\pi_{AB}^s(\Phi_A\Phi_B^*)$ since the functions in $\Phi_A$ and $\Phi_B$ are supported on disjoint coordinate sets. 

\paragraph{2. Initial and terminal constraints.}
The initial and terminal constraints \eqref{constraint-endpoint Mar} of the marginal relaxation are further relaxed at the moment level to hold only for one-cluster moments and retained cluster-pair moments. Explicitly, we have
\begin{equation}\label{constraint-endpoint Mom}
    \begin{split}
        (M_A^0)_{\rm L}=\mu_A\left({\rm R}_x(\Phi_A\Phi_A^*)\right), & \quad (M_A^{T-1})_{\rm R}=\nu_A\left({\rm R}_y(\Phi_A\Phi_A^*)\right), \quad A\in\C \\[1ex]
        (M_{AB}^0)_{\rm L}=\mu_{AB}\left({\rm R}_x(\Phi_A\Phi_B^*)\right), & \quad (M_{AB}^{T-1})_{\rm R}=\nu_{AB}\left({\rm R}_y(\Phi_A\Phi_B^*)\right), \quad AB\in\E.
    \end{split}
\end{equation}

\paragraph{3. Mass conservation constraints.}
The mass conservation constraints \eqref{constraint-time consis Mar} are relaxed to the moment level for all clusters $A \in \C$ and retained cluster pairs $AB \in \E$:
\begin{equation}\label{constraint-time consis Mom}
    \begin{split}
        (M_A^s)_{\rm R} & =(M_A^{s+1})_{\rm L}, \quad A\in\C, \quad s=0,\ldots,T-2 \\
        (M_{AB}^s)_{\rm R} & =(M_{AB}^{s+1})_{\rm L}, \quad AB\in\E, \quad s=0,\ldots,T-2.
    \end{split}
\end{equation}

\paragraph{4. Markov kernel constraints.}
The Markov kernel constraints \eqref{constraint-control Mar} are imposed at the moment level through a convex set $\cD^s_{\mom}$:
\begin{align}\label{constraint-control Mom}
    \{M_A^s,M_{AB}^s\}_{A \in \C, \, AB \in [\C]_2}\in\cD^s_{\mom}, \quad s=0,\ldots,T-1. 
\end{align}
This is the moment-level relaxation of the admissible transition constraint $\pi^s\in\cD^s$. For instance, if coordinate $i$ is required to remain frozen at step $s$, then the support constraint $x_i-y_i=0$ can be imposed by linear equations of the form
\begin{align*}
    \int (x_i-y_i)q(z)\dd \pi^s(z)=0
\end{align*}
for all retained test functions $q$ such that $(x_i-y_i)q$ lies in the span of the retained product basis.

\paragraph{5. Local positivity constraints.}
The local positivity constraints \eqref{constraint-local positivity Mar} require that, for each pair $AB \in [\C]_2$, the triple $(M_A^s, M_B^s, M_{AB}^s)$ is compatible with a nonnegative measure on $Z_A \times Z_B$. We impose them through a convex relaxation of local moment representability. For each $AB\in[\C]_2$, let $\mathcal{K}_{AB}^{\mathrm{loc}}$ denote the chosen convex set of admissible local moment triples. We impose
\begin{align}\label{constraint-local positivity Mom}
    (M_A^s,M_B^s,M_{AB}^s)\in\mathcal{K}_{AB}^{\mathrm{loc}}, \quad AB\in[\C]_2, \quad s=0,\ldots,T-1.
\end{align}

\paragraph{6. Global positivity constraints.}
The global positivity constraints \eqref{constraint-PSD Mar} are relaxed by restricting the square-integrable test functions to the span of the cluster bases. Consider the function set $\mathcal{F}= \set{f= \sum_{A \in \C} v_A^*\Phi_A(z_A) \mid v_A \in \mathbb{C}^{|\Phi_A|}}$. If the moments are induced by a genuine coupling $\pi^s$, then, for every $f \in \mathcal{F}$, we have
\begin{align*}
    \sum_{A\in\C}v_A^* M_A^s v_A+2\mathrm{Re}\sum_{AB\in[\C]^2}v_A^* M_{AB}^s v_B = \pi^s\paren{|f|^2} \ge 0.
\end{align*}
This is equivalent to requiring the moment matrix defined in \eqref{def:Ms} to be positive semidefinite:
\begin{align}\label{constraint-PSD Mom}
    M^s\succeq 0, \quad s=0,\ldots,T-1.
\end{align}

Combining these constraints gives the cluster moment relaxation:
\begin{align}
    \min_{\set{M^s}_{s=0}^{T-1}} \; & \sum_{s=0}^{T-1} \bracket{\sum_{A\in\C}\langle C_{A}^s,M_A^s\rangle+\sum_{AB\in\E}\langle C_{AB}^s,M_{AB}^s\rangle} \label{problem:Mom} \tag{$\mom$}\\
    \st & M^s \text{ is consistent  \eqref{constraint-consis Mom}},
    \label{constraint-consis Mom-copy}\tag{$\mom$-a} \\
    & \text{initial and terminal constraints \eqref{constraint-endpoint Mom} for } \set{M_A^i, M_{AB}^i}_{\C, \, \E}, \; i = 0 \text{ or } T-1,
    \label{constraint-endpoint Mom-copy}\tag{$\mom$-b} \\
    & \text{mass conservation \eqref{constraint-time consis Mom} for } \set{M_A^s, M_{AB}^s}_{\C, \, \E},
    \label{constraint-time consis Mom-copy} \tag{$\mom$-c}\\
    & \text{Markov kernel constraints \eqref{constraint-control Mom} for } \set{M_A^s, M_{AB}^s}_{\C, \, [\C]_2},
    \label{constraint-control Mom-copy} \tag{$\mom$-d} \\
    & \text{local positivity \eqref{constraint-local positivity Mom}: }(M_A^s,M_B^s,M_{AB}^s)\in\mathcal{K}_{AB}^{\mathrm{loc}}, \quad AB\in[\C]_2,
    \label{constraint-local positivity Mom-copy}\tag{$\mom$-e} \\
    & \text{global positivity \eqref{constraint-PSD Mom}: } M^s\succeq 0.
    \label{constraint-PSD Mom-copy} \tag{$\mom$-f}
\end{align}
This is a semidefinite relaxation of the sequential-coupling problem. As in the marginal case, we obtain two further relaxations by dropping either the global positivity constraints $M^s \succeq 0$ or the local positivity constraints $(M_A^s,M_B^s,M_{AB}^s)\in \mathcal{K}_{AB}^{\mathrm{loc}}$.

\begin{align}
    \min_{\substack{\set{M^s_A}_\C \\ \set{M_{AB}^s}_\E}} \; & \sum_{s=0}^{T-1} \bracket{\sum_{A\in\C}\langle C_{A}^s,M_A^s\rangle+\sum_{AB\in\E}\langle C_{AB}^s,M_{AB}^s\rangle} 
    \label{problem:Mom1} \tag{$\mom^1$}\\
    \st & M^s \text{ is consistent \eqref{constraint-consis Mom}},
    \label{constraint-consis Mom1} \tag{$\mom^1$-a}\\
    & \text{initial and terminal constraints \eqref{constraint-endpoint Mom} for } \set{M_A^i, M_{AB}^i}_{\C, \, \E} ,\; i = 0 \text{ or } T-1,
    \label{constraint-endpoint Mom1}\tag{$\mom^1$-b}\\
    & \text{mass conservation \eqref{constraint-time consis Mom} for } \set{M_A^s, M_{AB}^s}_{\C, \, \E},
    \label{constraint-time consis Mom1} \tag{$\mom^1$-c}\\
    & \text{projected Markov kernel constraints: }\{M_A^s,M_{AB}^s\}_{\C, \, \E} \in \cD^s_{\mom,\E},
    \label{constraint-control Mom1} \tag{$\mom^1$-d}\\
    & \text{local positivity \eqref{constraint-local positivity Mom}: }(M_A^s,M_B^s,M_{AB}^s)\in\mathcal{K}_{AB}^{\mathrm{loc}}, \quad AB\in \E .
    \label{constraint-local positivity Mom1} \tag{$\mom^1$-e}
\end{align}
Problem \eqref{problem:Mom1} is obtained from \eqref{problem:Mom} by dropping the global positivity constraints and discarding the cross-moment variables $M^s_{AB}$ for $AB \notin \E$. The set $\cD_{\mom,\E}^s$ denotes the projection of the constraint set $\cD_{\mom}^s$ onto the retained variables. Dropping the local positivity constraints \eqref{constraint-local positivity Mom-copy} from \eqref{problem:Mom} yields the following relaxation:
\begin{align}
    \min_{\set{M^s}_{s=0}^{T-1}} \; & \sum_{s=0}^{T-1} \bracket{\sum_{A\in\C}\langle C_{A}^s,M_A^s\rangle+\sum_{AB\in\E}\langle C_{AB}^s,M_{AB}^s\rangle} \label{problem:Mom2} \tag{$\mom^2$}\\
    \st & M^s \text{ is consistent  \eqref{constraint-consis Mom}},
    \label{constraint-consis Mom2}\tag{$\mom^2$-a} \\
    & \text{initial and terminal constraints \eqref{constraint-endpoint Mom} for } \set{M_A^i, M_{AB}^i}_{\C, \, \E}, \; i=0 \text{ or } T-1,
    \label{constraint-endpoint Mom2}\tag{$\mom^2$-b} \\
    & \text{mass conservation \eqref{constraint-time consis Mom} for } \set{M_A^s, M_{AB}^s}_{\C, \, \E},
    \label{constraint-time consis Mom2} \tag{$\mom^2$-c}\\
    & \text{Markov kernel constraints \eqref{constraint-control Mom} for } \set{M_A^s, M_{AB}^s}_{\C, \, [\C]_2},
    \label{constraint-control Mom2} \tag{$\mom^2$-d} \\
    & \text{global positivity \eqref{constraint-PSD Mom}: } M^s\succeq 0.
    \label{constraint-PSD Mom2} \tag{$\mom^2$-e}
\end{align}

\section{Recovering Benamou--Brenier dynamics from problem \eqref{problem:coupling}}\label{sec:3}

In this section, we show that the sequential-coupling problem \eqref{problem:coupling} contains the Benamou--Brenier problem \eqref{problem:BB} as a special case. We work on a convex set $\Omega\subset\R^d$, and let $\mu, \nu \in \mathcal{P}_p(\Omega)$, where $\mathcal{P}_p(\Omega)$ denotes the set of probability measures on $\Omega$ with finite $p$-th moment. Fix a time grid $0 = t_0 < t_1 < \cdots < t_T =1$. For $1\le p<\infty$, choose the one-step cost
\begin{align}\label{eq:dot-cost 3.1}
    c^s(x,y)=\frac{\|x-y\|^p}{(\Delta t_s)^{p-1}}, \quad \Delta t_s=t_{s+1}-t_s.
\end{align}
We impose no Markov kernel constraints. In the notation of problem \eqref{problem:coupling}, this means $\cD^s = \mathcal{P}(\Omega \times \Omega)$ in \eqref{constraint-control coupling}. With these choices, problem \eqref{problem:coupling} becomes:
\begin{align}
    \inf_{\set{\pi^s}_{s=0}^{T-1}} \;  &   \sum_{s=0}^{T-1} \int_{\Omega\times\Omega} \frac{\|x-y\|^p}{(\Delta t_s)^{p-1}} \dd \pi^s(x,y) \label{problem:dot} \\
    \st & \PL \pi^0 = \mu, \quad \PR \pi^{T-1} = \nu, \label{constraint-endpoint dot} \tag{3.2-a}\\
    & \PR \pi^s = \PL \pi^{s+1}, \quad s=0,\ldots,T-2, \label{constraint-time consis dot} \tag{3.2-b}\\
    & \pi^s \ge 0, \quad s=0,\ldots,T-1. \label{constraint-positivity dot} \tag{3.2-c}
\end{align}

This section shows how solving \eqref{problem:dot} recovers an optimal solution of the Benamou--Brenier problem \eqref{problem:BB}. In \Cref{subsec:3.1}, we prove that these two problems have the same optimal value and that a primal solution of \eqref{problem:dot} yields the Benamou--Brenier measure curve at the prescribed grid times. In \Cref{subsec:3.2}, under a standard connected-support condition, we show that a dual solution of \eqref{problem:dot} recovers the Benamou--Brenier velocity field at the prescribed grid times. Finally, in \Cref{subsec:3.3}, we explain how to extract an approximate velocity field from the moment relaxation for \eqref{problem:dot}.

\subsection{Recovering the measure curve in \eqref{problem:BB} from primal solution of \eqref{problem:dot}}\label{subsec:3.1}
In this subsection, \Cref{prop:3.1} shows that an optimal solution of \eqref{problem:dot} recovers the values of a Benamou--Brenier measure curve at the prescribed grid times. We also show that \eqref{problem:dot} has the same optimal value as the Benamou--Brenier problem \eqref{problem:BB}. 

To prove the result, we first introduce the intermediate distributions induced by a feasible sequence of couplings. For any feasible sequence $\set{\pi^s}_{s=0}^{T-1}$, define
\begin{align*}
    \rho^0 \coloneqq \PL \, \pi^0, \quad \rho^{s+1} \coloneqq \PR \, \pi^s, \quad s = 0, \ldots, T-1.
\end{align*}
By the initial and terminal constraints \eqref{constraint-endpoint dot} and the mass-conservation constraints \eqref{constraint-time consis dot}, we have
\begin{align*}
    \rho^0 = \mu, \quad \rho^T = \nu, \quad \pi^s \in \Pi(\rho^s, \rho^{s+1}), \quad s = 0, \ldots, T-1,
\end{align*}
where $\Pi(\rho^s, \rho^{s+1})$ denotes the set of all couplings between $\rho^s$ and $\rho^{s+1}$. Conversely, any sequence $\set{\rho^s}_{s=0}^T$ with $\rho^0 = \mu$ and $\rho^T = \nu$, together with couplings $\pi^s\in \Pi(\rho^s, \rho^{s+1})$, defines a feasible solution of \eqref{problem:dot}. Therefore, minimizing \eqref{problem:dot} over the sequence of couplings is equivalent to first choosing the intermediate marginals at the prescribed grid times $\set{\rho^s}_{s=0}^T$, and then choosing an optimal coupling between each pair of consecutive marginals. 

For fixed $\rho^{s}$ and $\rho^{s+1}$, minimizing over $\pi^s \in \Pi(\rho^s, \rho^{s+1})$ gives
\begin{align*}
    \inf_{\pi^s\in\Pi(\rho^s,\rho^{s+1})}\int_{\Omega\times\Omega}\frac{\|x-y\|^p}{(\Delta t_s)^{p-1}}\dd \pi^s(x,y)=\frac{1}{(\Delta t_s)^{p-1}} \W_p^p(\rho^s,\rho^{s+1}).
\end{align*}
Thus \eqref{problem:dot} is equivalent to the reduced form
\begin{align}\label{problem:dot'}
    \inf_{\set{\rho^s \in \mathcal{P}_p(\Omega)}_{s=0}^T}\; \sum_{s=0}^{T-1}\frac{1}{(\Delta t_s)^{p-1}}\W_p^p(\rho^s,\rho^{s+1}),\qquad \rho^0=\mu,\qquad \rho^T=\nu. \tag{3.2'}
\end{align}
We call an optimal measure curve $\{q(\plh, t)\}_{t\in[0,1]}$ in \eqref{problem:BB} a Benamou--Brenier geodesic between $\mu$ and $\nu$. The following proposition explains how to recover the Benamou--Brenier geodesic from the solution of \eqref{problem:dot}.

\begin{prop}\label{prop:3.1}
    Let $1\le p<\infty$. Let $\Omega \subseteq \R^d$ be convex, and let $\mu, \nu \in \mathcal{P}_p(\Omega)$. The problem \eqref{problem:dot}, equivalently the reduced form \eqref{problem:dot'}, has optimal value $\W_p^p(\mu,\nu)$. Explicitly:
    \begin{equation}\label{eq:prop 3.1-part 1}
        \min_{\set{\pi_s}_{s=0}^{T-1}} \; \sum_{s = 0}^{T-1} \int_{\Omega \times \Omega} \frac{\|x-y\|^p}{\Delta t_s^{p-1}} \dd \pi^s(x,y) = \W_p^p(\mu, \nu)
    \end{equation}
    where the infimum is subject to constraints \eqref{constraint-endpoint dot}--\eqref{constraint-positivity dot}. Moreover, if $\set{q(\cdot, t),v(\cdot, t)}_{t \in [0,1]}$ is an optimal solution of  \eqref{problem:BB}, then the grid distributions
    \begin{align}
        \rho^s \coloneqq q(\cdot, t_s), \quad s=0,\ldots,T,
    \end{align}
    solve the reduced problem \eqref{problem:dot'}. Conversely, any minimizer $\set{\rho^s}_{s=0}^T$ of problem \eqref{problem:dot'} can be connected on each interval $[t_s,t_{s+1}]$ by Benamou--Brenier geodesics to obtain an optimal solution of the problem \eqref{problem:BB}.
\end{prop}

\begin{proof}
    Let 
    \[
        D_T:= \min_{\set{\rho^s \in \mathcal{P}_p(\Omega)}_{s=0}^T} \sum_{s=0}^{T-1} \frac{1}{(\Delta t_s)^{p-1}} \W_p^p(\rho^s,\rho^{s+1}), \quad \rho^0=\mu,\quad \rho^T=\nu,
    \]
    be the optimal value of \eqref{problem:dot'}. We prove that $D_T=\W_p^p(\mu, \nu)$. 
    
    First we show $D_T\leq \W_p^p(\mu, \nu)$. Let $(q,v)$ be
    any admissible pair for the Benamou--Brenier problem, and set
    \[
        \rho^s:=q(\,\cdot \,, t_s), \quad s=0,\ldots,T.
    \]
    Fix $s\in\{0,\ldots,T-1\}$. On the interval $[t_s,t_{s+1}]$, define
    the rescaled curve on $[0,1]$ by
    \[
        \widetilde q^s(\, \plh ,\tau) = q(\, \plh , t_s+\tau\Delta t_s), \quad \widetilde v^s(\, \plh, \tau) = \Delta t_s\,v( \, \plh, t_s+\tau\Delta t_s), \quad \tau\in[0,1].
    \]
    Writing $t=t_s+\tau\Delta t_s$, we obtain
    \[
        \partial_\tau \widetilde q^s = \Delta t_s\,\partial_t q_t = -\Delta t_s\,\nabla\cdot(q \, v) = -\nabla\cdot(\widetilde q^s \widetilde v^s).
    \]
    Hence $(\widetilde q^s,\widetilde v^s)$ satisfies the continuity equation on $[0,1]$. The initial and terminal values of $\widetilde q^s$ are 
    \[
        \widetilde q^s (\, \plh, 0)=\rho^s, \quad \widetilde q^s(\, \plh, 1)=\rho^{s+1}.
    \]
    By the Benamou--Brenier formula on the unit interval,
    \[
        \W_p^p(\rho^s,\rho^{s+1}) \leq \int_0^1\int_{\Omega} \|\widetilde v^s\|^p \dd \widetilde q^s\dd \tau.
    \]
    Using the definition of $\widetilde v^s$ and the change of variables $t=t_s+\tau\Delta t_s$, we get
    \[
        \int_0^1\int_{\Omega} \|\widetilde v^s\|^p \dd \widetilde q^s \dd \tau = (\Delta t_s)^{p-1} \int_{t_s}^{t_{s+1}}\int_{\Omega} \|v\|^p\dd q\dd t.
    \]
    Therefore,
    \[
        \frac{1}{(\Delta t_s)^{p-1}} \W_p^p(\rho^s,\rho^{s+1}) \leq \int_{t_s}^{t_{s+1}}\int_{\Omega} \|v\|^p\dd q\dd t.
    \]
    Summing over $s = 0, \ldots, T-1$ and taking the infimum over all admissible Benamou--Brenier pairs $(q,v)$ gives $\cD_T\leq \W_p^p(\mu, \nu)$. 
    
    Conversely, let $\rho^0,\ldots,\rho^T$ be any admissible sequence for problem \eqref{problem:dot'}. For each $s=0,\ldots,T-1$, choose an optimal Benamou--Brenier pair on the unit interval connecting $\rho^s$ to $\rho^{s+1}$. Denote this pair by $(\widetilde q^s,\widetilde v^s)$. Thus
    \[
        \widetilde q^s(\, \plh, 0)=\rho^s,\quad \widetilde q^s(\, \plh, 1)=\rho^{s+1},
    \]
    and
    \[
        \int_0^1\int_{\Omega} \|\widetilde v^s\|^p \dd \widetilde q^s \dd \tau = \W_p^p(\rho^s,\rho^{s+1}).
    \]
    Rescale this optimal pair to the interval $[t_s,t_{s+1}]$ by setting
    \[
        q(\, \plh, t) = \widetilde q^s(\,\plh, \frac{t-t_s}{\Delta t_s}), \quad v(\, \plh, t) = \frac{1}{\Delta t_s} \widetilde v^s(\, \plh, \frac{t-t_s}{\Delta t_s}), \quad t\in[t_s,t_{s+1}].
    \]
    Then $(q, v)$ satisfies the continuity equation on $[t_s,t_{s+1}]$, with initial and terminal values $\rho^s$ and $\rho^{s+1}$. Its action on this interval is
    \[
        \int_{t_s}^{t_{s+1}}\int_{\Omega} \|v\|^p\dd q \dd t = \frac{1}{(\Delta t_s)^{p-1}} \W_p^p(\rho^s,\rho^{s+1}).
    \]
    Concatenating these rescaled optimal pairs over all intervals $[t_s,t_{s+1}]$ gives an admissible Benamou--Brenier pair from $\mu$ to $\nu$, with total action
    \[
        \sum_{s=0}^{T-1} \frac{1}{(\Delta t_s)^{p-1}} \W_p^p(\rho^s,\rho^{s+1}).
    \]
    Since the admissible sequence $\rho^0,\ldots,\rho^T$ was arbitrary, taking the infimum over all such sequences gives $\W_p^p(\mu,\nu) \leq D_T$. Therefore $D_T= \W_p^p(\mu,\nu)$. This proves \eqref{eq:prop 3.1-part 1}.

    The statement about minimizers follows from the same two constructions. If $(q,v)$ is optimal in \eqref{problem:BB}, then the grid values $\rho^s=q(\plh, t_s)$ attain $\W_p^p(\mu,\nu)$ in \eqref{problem:dot'}. Hence they solve \eqref{problem:dot'}. Conversely, if $\{\rho^s\}_{s=0}^T$ is optimal in \eqref{problem:dot'}, concatenating optimal Benamou--Brenier geodesics between the consecutive distributions gives a continuous-time admissible pair with total action $\W_p^p(\mu,\nu)$, hence an optimal Benamou--Brenier pair. This completes the proof.
\end{proof}
Thus, with $\cD^s = \mathcal{P}(\Omega \times \Omega)$ and the cost defined in \eqref{eq:dot-cost 3.1}, problem \eqref{problem:coupling} recovers a Benamou--Brenier geodesic at the prescribed grid times and introduces no time-discretization error in the optimal value. Applying the marginal or moment relaxations from \Cref{sec:conv relaxation} to this formulation then yields low-order statistics of the intermediate distributions and a lower bound for $\W_p^p$.

\subsection{Recovering the velocity in \eqref{problem:BB} from dual solution of \eqref{problem:dot}}\label{subsec:3.2}

In this subsection, we derive the dual of problem \eqref{problem:dot} and show that, under a standard connected-support condition, its dual potentials encode the Benamou–Brenier velocity field at the grid times. We restrict the presentation to the case $p=2$. Analogous formulas can be derived for $1 \le p < \infty$. The main statement is as follows.

\begin{prop}\label{prop:3.2}
Assume $p=2$. Let $\lambda^0$ and $\lambda^T$ be the dual variables associated with the initial and terminal constraints \eqref{constraint-endpoint dot}, respectively. For $1 \le s \le T-1$, let $\lambda^s$ be the dual variable associated with the mass conservation constraints \eqref{constraint-time consis dot}. The dual problem of \eqref{problem:dot} can be written as
    \begin{align}
        \sup_{\{\lambda^s\}_{s=0}^T} \,& \int_{\Omega}\lambda^T(x)\dd \nu(x)-\int_{\Omega}\lambda^0(x)\dd \mu(x) \label{problem:dual dot} \tag{dual 3.2}\\
        \textup{s.t. } \; & \frac{\|x-y\|^2}{\Delta t_s} - \lambda^{s+1}(y) + \lambda^s(x) \ge 0, \quad s=0,\ldots,T-1. \label{constraint-dual dot}\tag{dual 3.2-a}
    \end{align}
    Here the supremum is taken over all measurable functions $\lambda^s:\Omega\to\R$, with $\lambda^0\in L^1(\mu)$ and $\lambda^T\in L^1(\nu)$. 
    
    Moreover, let $(q,v)$ be an optimal pair for the Benamou–Brenier problem \eqref{problem:BB}, and set
    \begin{align*}
        \rho^s \coloneqq q(\plh, t_s), \quad s = 0,\ldots, T.
    \end{align*}
    Assume that the Benamou--Brenier dual problem (whose precise formulation is deferred to \Cref{lem:3.3}) admits a $C^1$ optimal solution and that each $\rho^s$ is absolutely continuous, with connected support and Lebesgue-negligible boundary. Then any two optimal dual families of \eqref{problem:dual dot} differ by a common additive constant at every time level, $\rho^s$-almost everywhere. After imposing the normalization $\int_\Omega \lambda^0\dd \mu = 0$, the optimal dual family is unique $\rho^s$-almost everywhere. Moreover, the optimal velocity field may be chosen so that, for $\rho^s$-almost every $x$,
    \begin{align}\label{eq:velocity recovery}
        v(x, t_s)=\frac{1}{2}\nabla\lambda^s(x),  \quad s=0,\ldots,T. 
    \end{align}
\end{prop}
Thus, in this special case under the above assumptions, we can recover the Benamou–Brenier velocity field at the grid times from the gradients of the dual potentials. When this velocity field generates a flow map, the associated Markov kernel in problem \eqref{problem:1.1} is obtained from the induced transition along the flow.

Before proving \Cref{prop:3.2}, we recall the standard Benamou--Brenier duality result.

\begin{lem}\label{lem:3.3}
    Let $\varphi$ be the dual variable associated with the continuity equation constraint in the Benamou--Brenier problem \eqref{problem:BB}.
    The dual problem of  \eqref{problem:BB} can be written as
    \begin{align}
        \sup_{\varphi} \quad & \int_{\Omega}\varphi(x,1)\dd \nu(x)-\int_{\Omega}\varphi(x,0)\dd \mu(x)\label{problem:dual BB} \tag{dual BB} \\
        \textup{s.t.} \quad & \partial_t\varphi(x,t)+\frac{1}{4}\|\nabla\varphi(x,t)\|^2\le 0. \label{constraint dual BB} \tag{dual BB-a}
    \end{align}
    Here the supremum is taken over all $C^1$ space-time functions $\varphi:\Omega\times[0,1]\to\R$ whose initial and terminal values satisfy $\varphi(\cdot,0)\in L^1(\mu)$ and $\varphi(\cdot,1)\in L^1(\nu)$. Moreover, if $(q,v)$ and $\varphi$ form an optimal primal-dual pair, then
    \begin{align}\label{eq:v=1/2grad varphi}
        v(x,t)=\frac{1}{2}\nabla\varphi(x,t)
    \end{align}
    for $q(\plh, t)$-almost every $x$ and almost every $t \in [0,1]$. 
\end{lem}
\begin{proof}
    This is the standard duality and optimality condition for the Benamou--Brenier formulation; see, e.g., \cite[Section 6.1]{Santambrogio2015}.
\end{proof}

We now prove \Cref{prop:3.2}.

\begin{proof}[Proof of \Cref{prop:3.2}]
    Using the dual variables $\set{\lambda^s}_{s=0}^T$ specified in the statement, the Lagrangian of the problem~\eqref{problem:dot} is
    \begin{align*}
        \mathcal{L}(\{\pi^s\},\{\lambda^s\}) = & \sum_{s=0}^{T-1}\int_{\Omega\times\Omega} \left[\frac{\|x-y\|^2}{\Delta t_s} -\lambda^{s+1}(y) + \lambda^s(x)\right]\dd \pi^s(x,y) \\
        & - \int_{\Omega}\lambda^0(x)\dd \mu(x)+\int_{\Omega}\lambda^T(x)\dd \nu(x). 
    \end{align*}
    Taking the infimum over nonnegative measures $\pi^s$ gives a finite value if and only if
    \begin{align*}
        \frac{\|x-y\|^2}{\Delta t_s} - \lambda^{s+1}(y) + \lambda^s(x) \ge 0,\quad s=0,\ldots,T-1. 
    \end{align*}
    Under this condition, the infimum of $\mathcal{L}$ over $\{\pi^s\}$ equals
    \begin{align*}
        \int_{\Omega}\lambda^T(x)\dd \nu(x)-\int_{\Omega}\lambda^0(x)\dd \mu(x).
    \end{align*}
    Therefore the dual of \eqref{problem:dot} is precisely \eqref{problem:dual dot}.

    It remains to prove the uniqueness and velocity recovery statements. Let $\varphi(x,t)$ be an optimal solution of the dual problem of \eqref{problem:BB} in \Cref{lem:3.3}, and define
    \begin{align*}
        \lambda^s(x) \coloneqq \varphi(x, t_s), \quad s = 0, \ldots, T.
    \end{align*}
    Since a common additive shift leaves both the constraints and the objective of \eqref{problem:dual dot} unchanged, we shift the entire family $\set{\lambda^s}_{s=0}^T$, if necessary, so that
    \begin{align*}
        \int_\Omega \lambda^0(x) \dd \mu (x) = 0.
    \end{align*}
    We first show that these grid potentials solve \eqref{problem:dual dot}. Fix $s \in \set{0, \ldots, T-1}$, and let $\gamma:[t_s,t_{s+1}]\to\Omega$ be any absolutely continuous curve satisfying $\gamma(t_s)=x$ and $\gamma(t_{s+1})=y$. By the constraint \eqref{constraint dual BB},
    \begin{align*}
        \frac{\dd}{\dd t}\varphi(\gamma(t),t)=\partial_t\varphi+\nabla\varphi\cdot\dot{\gamma}\le -\frac{1}{4}\|\nabla\varphi\|^2+\nabla\varphi\cdot\dot{\gamma}\le \|\dot{\gamma}(t)\|^2.
    \end{align*}
    Integrating over $[t_s,t_{s+1}]$ gives
    \begin{align*}
        \varphi(y,t_{s+1})-\varphi(x,t_s)\le \int_{t_s}^{t_{s+1}}\|\dot{\gamma}(t)\|^2\dd t.
    \end{align*}
    Minimizing over all such curves, and using the convexity of $\Omega$, the minimizing curve is the straight line from $x$ to $y$. Hence
    \begin{align*}
        \varphi(y,t_{s+1})-\varphi(x,t_s)\le \frac{\|x-y\|^2}{\Delta t_s}.
    \end{align*}
    Therefore,
    \begin{align*}
        \frac{\|x-y\|^2}{\Delta t_s}-\lambda^{s+1}(y)+\lambda^s(x) \ge 0,
    \end{align*}
    so the grid potentials $\set{\lambda^s}_{s=0}^T$ are feasible for \eqref{problem:dual dot}. Their objective value is
    \begin{align*}
        \int_{\Omega}\lambda^{T}(x) \dd \nu(x)-\int_{\Omega}\lambda^0(x)\dd \mu(x) = \int_{\Omega}\varphi(x,1)\dd \nu(x)-\int_{\Omega}\varphi(x,0)\dd \mu(x),
    \end{align*}
    which is the optimal value of \eqref{problem:dual BB}. By \Cref{lem:3.3} and \Cref{prop:3.1}, this value equals the optimal value of \eqref{problem:dot}. Hence $\set{\lambda^s}_{s=0}^T$ is an optimal solution of \eqref{problem:dual dot}.

    We next establish the uniqueness of the optimal solution of \eqref{problem:dual dot}. Let $\{\widetilde{\lambda}^s\}_{s=0}^T$ be any other normalized optimal solution. For each $s$, dual feasibility gives
    \begin{align*}
        \int_{\Omega}\widetilde{\lambda}^{s+1}\dd \rho^{s+1}-\int_{\Omega}\widetilde{\lambda}^s\dd \rho^s\le \frac{1}{\Delta t_s} \W_2^2(\rho^s,\rho^{s+1}).
    \end{align*}
    Summing over $s$, the left-hand side telescopes to the optimal value of \eqref{problem:dual dot}, while the right-hand side equals $\W_2^2(\mu,\nu)$ by \Cref{prop:3.1}. Hence equality holds for every $s$, and both $(-\lambda^s,\lambda^{s+1})$ and $(-\widetilde{\lambda}^s,\widetilde{\lambda}^{s+1})$ are optimal dual pairs for the quadratic transport problem between $\rho^s$ and $\rho^{s+1}$. By the standard uniqueness result for Kantorovich potentials \parencite[Corollary 2.7]{delbarrio2021clt}, applied to each one-step problem and its reverse, together with equality of the corresponding one-step dual values, there is a constant $C_s$ such that
    \begin{align*}
        \lambda^s=\widetilde{\lambda}^s+C_s,\quad \lambda^{s+1}=\widetilde{\lambda}^{s+1}+C_s,
    \end{align*}
    on the corresponding marginal supports. Since $\lambda^{s+1}$ is shared by two adjacent one-step problems, $C_s=C_{s+1}$. Thus all constants are equal to a single common constant $C$. The normalization at $s=0$ gives $C=0$, proving uniqueness of the normalized dual family. 

    Consequently, 
    \begin{align*}
        \nabla \widetilde{\lambda}^s = \nabla \lambda^s = \nabla\varphi(\plh, t_s)
    \end{align*}
    for $\rho^s$-almost every $x$. By \Cref{lem:3.3}, the optimal Benamou--Brenier velocity field satisfies $v(x,t) = \frac{1}{2}\nabla \varphi(x,t)$ for almost every $t$. Since changing the velocity at the finitely many grid times does not affect either the continuity equation or the Benamou--Brenier action, we choose its values at these times so that
    \begin{align*}
        v(x,t_s) = \frac{1}{2}\nabla \varphi(x,t_s) = \frac{1}{2} \nabla \lambda^s(x), \quad s = 0, \ldots, T.
    \end{align*}
    This proves \eqref{eq:velocity recovery} and completes the proof.
\end{proof}

\subsection{Approximating the velocity in \eqref{problem:BB} from moment relaxation for \eqref{problem:dot}}\label{subsec:3.3}
\Cref{subsec:3.2} shows how to recover the Benamou–Brenier velocity field from the dual potentials of problem \eqref{problem:dot} in the quadratic case $p=2$. In this subsection, we explain how to obtain an approximation of the velocity field from a dual solution of the cluster moment relaxation. We focus on the simplified relaxation \eqref{problem:Mom2}, where the local positivity constraints are dropped while the global positivity constraints are retained.

The relaxed dual problem may have more than one optimizer. Throughout this subsection, the phrase \emph{the dual solution} refers to any fixed optimal solution of the relaxed dual problem. The construction below is applied to this optimizer, and its common additive constant has no effect on the recovered velocity field.

We derive the dual of the moment relaxation \eqref{problem:Mom2} for problem \eqref{problem:dot} directly from a measure-level formulation. Let
\begin{align*}
    \mathcal{Z}=\Omega\times\Omega,\quad z=(x,y)\in \mathcal{Z},\quad s=0,\ldots,T-1.
\end{align*}
Instead of optimizing over moment matrices, the relaxation can be viewed as optimizing over signed measures $\eta^s\in\mathcal{M}(\mathcal{Z})$. Let $\Phi(z)=(\Phi_A(z_A))_{A\in\mathcal{C}}$ denote the vector of retained cluster basis functions. Define the associated finite-dimensional SOS cone
\begin{align*}
    \Sigma^s:=\{\langle S^s,\Phi(z)\Phi(z)^*\rangle:S^s\succeq 0\}.
\end{align*}
Equivalently, $\Sigma^s$ consists of functions of the form $\sigma(z)=\sum_{\ell}|q_{\ell}(z)|^2$, where $q_{\ell}\in\operatorname{span}\{\Phi\}$. Let 
\begin{align*}
    \mathcal{W}_\G \coloneqq  \operatorname{span}\set{ {\rm R}_x \paren{\set{\Phi_A\Phi_A^*: A \in \C} \cup \set{\Phi_A \Phi_B^*: AB \in \E}}},
\end{align*}
where ${\rm R}_x$ selects the entries that depend only on the $x$-copy of $z = (x,y)$. We can write the measure-level relaxation as
\begin{align}
    \inf_{\{\eta^s\}_{s=0}^{T-1}}\quad & \sum_{s=0}^{T-1}\langle c^s,\eta^s\rangle 
    \label{problem:measure mom2} \\
    \st & \langle f \circ {\rm P}_{\rm L},\eta^0\rangle = \langle f,\mu\rangle, \quad \forall f\in\mathcal{W}_\G, \tag{3.7-a}\\
    & \langle f \circ {\rm P}_{\rm R},\eta^{T-1}\rangle=\langle f,\nu\rangle,\quad \forall f\in\mathcal{W}_\G, \tag{3.7-b}\\
    & \langle f\circ {\rm P}_{\rm R},\eta^s\rangle = \langle f \circ {\rm P}_{\rm L},\eta^{s+1}\rangle, \quad \forall f \in\mathcal{W}_\G,\quad s=0,\ldots,T-2, \tag{3.7-c}\\
    & \langle\sigma,\eta^s\rangle\ge 0,\quad \forall\sigma\in\Sigma^s,\quad s=0,\ldots,T-1. \label{constraint-positivity measure mom2 dot}\tag{3.7-d}
\end{align}
This formulation is equivalent to the moment-matrix relaxation \eqref{problem:Mom2} applied to problem \eqref{problem:dot}. Indeed, given signed measures $\eta^s$, define
\[
    M^s=\int_{\mathcal{Z}}\Phi(z)\Phi(z)^*\dd \eta^s(z).
\]
Then, for any $\sigma(z)=\langle S^s,\Phi(z)\Phi(z)^*\rangle\in\Sigma^s$, we have
\begin{align} \label{eq:measure-matrix}
    \langle\sigma,\eta^s\rangle=\left\langle S^s,\int_{\mathcal{Z}}\Phi(z)\Phi(z)^*\dd \eta^s(z)\right\rangle=\langle S^s,M^s\rangle.
\end{align}
Therefore $\angles{\sigma,\eta^s} \ge 0$ for any $\sigma\in\Sigma^s$ is equivalent to $M^s\succeq 0$. Conversely, let $M^s$ be a feasible moment matrix for relaxation \eqref{problem:Mom2} for problem \eqref{problem:dot}. The consistency constraints imply that $M^s$ defines a well-defined linear functional $L^s$ on the finite-dimensional space $\mathcal{V}^s=\operatorname{span}\{\text{entries of } \Phi(z)\Phi(z)^*\}$ after identifying repeated product functions. Since $\mathcal{V}^s$ is finite dimensional, any such linear functional can be represented on $\mathcal{V}^s$ by a signed atomic measure: there exist points $z_\ell \in \mathcal{Z}$ and weights $w_\ell \in \R$ such that $\eta^s=\sum_{\ell}w_{\ell}\delta_{z_{\ell}}$ and
\begin{align*}
    \int f\dd \eta^s=L^s(f),\quad \forall f\in\mathcal{V}^s, \quad s = 0, \ldots, T-1.
\end{align*}
In particular, 
\[
    M^s=\int \Phi(z)\Phi(z)^*\dd \eta^s(z).
\]
Thus the measure formulation \eqref{problem:measure mom2} and the matrix formulation \eqref{problem:Mom2} applied to problem \eqref{problem:dot} have the same feasible retained moments and the same optimal value.

We now derive the dual of \eqref{problem:measure mom2}. Let
\[
    \widehat{\lambda}^0, \widehat{\lambda}^T \in \mathcal{W}_\G; \qquad \widehat{\lambda}^s \in \mathcal{W}_\G, \quad s = 1, \ldots, T-1
\]
be the dual potentials associated with the endpoint and mass-conservation constraints, respectively, and let
\[
    \tau^s \in \Sigma^s, \quad s = 0, \ldots, T-1
\]
be the multipliers associated with the relaxed positivity constraints \eqref{constraint-positivity measure mom2 dot}. The Lagrangian of the problem~\eqref{problem:measure mom2} is
\begin{align*}
    \mathcal{L} = &\sum_{s=0}^{T-1}\langle c^s,\eta^s\rangle + \langle\widehat{\lambda}^0\circ {\rm P}_{\rm L},\eta^0\rangle-\langle\widehat{\lambda}^0,\mu\rangle +\langle\widehat{\lambda}^T , \nu \rangle-\langle\widehat{\lambda}^T\circ {\rm P}_{\rm R} , \eta^{T-1}\rangle \\
     & +\sum_{s=0}^{T-2}\left[\langle\widehat{\lambda}^{s+1}\circ {\rm P}_{\rm L},\eta^{s+1}\rangle-\langle\widehat{\lambda}^{s+1}\circ {\rm P}_{\rm R},\eta^s\rangle\right] -\sum_{s=0}^{T-1}\langle\tau^s,\eta^s\rangle \\
    = & \langle\widehat{\lambda}^T,\nu\rangle-\langle\widehat{\lambda}^0,\mu\rangle+\sum_{s=0}^{T-1}\left\langle c^s+\widehat{\lambda}^s\circ {\rm P}_{\rm L}-\widehat{\lambda}^{s+1}\circ {\rm P}_{\rm R}-\tau^s,\eta^s\right\rangle. 
\end{align*}
Since $\eta^s$ is a signed measure, the infimum of $\mathcal{L}$ over $\eta^s$ is finite only if
\begin{align*}
    c^s(x,y)+\widehat{\lambda}^s(x)-\widehat{\lambda}^{s+1}(y) = \tau^s(x,y), \quad s=0,\ldots,T-1.
\end{align*}
Substituting the quadratic Benamou--Brenier cost $c^s(x,y) = \|x-y\|^2 / \Delta t_s$, the dual problem is
\begin{align}
    \sup_{\widehat{\lambda}^s\in \mathcal{W_\G}, \, \tau^s \in\Sigma^s} \quad & \int_{\Omega}\widehat{\lambda}^T(x)\dd \nu(x)-\int_{\Omega}\widehat{\lambda}^0(x)\dd \mu(x) \tag{dual 3.7}\\
    \st & \frac{\|x-y\|^2}{\Delta t_s}-\widehat{\lambda}^{s+1}(y)+\widehat{\lambda}^s(x)=\tau^s(x,y),\quad s=0,\ldots,T-1.  \label{conssos} \tag{dual 3.7-a}
\end{align}
Using the Gram representation of $\Sigma^s$, the constraint \eqref{conssos} can be written as
\begin{align}\label{eq:SOS identity}
    \frac{\|x-y\|^2}{\Delta t_s} - \widehat{\lambda}^{s+1}(y)+\widehat{\lambda}^s(x)=\langle S^s,\Phi(z)\Phi(z)^*\rangle,\qquad S^s\succeq 0,\qquad s=0,\ldots,T-1.
\end{align}
This shows that the dual of the moment relaxation \eqref{problem:measure mom2} for problem \eqref{problem:dot} is an SOS relaxation of the exact dual problem \eqref{problem:dual dot}. Indeed, the pointwise nonnegativity constraint in the exact dual problem \eqref{constraint-dual dot} is replaced by the stronger SOS certificate in the retained cone $\Sigma^s$. Consequently, the dual potentials $\widehat{\lambda}^s$ should be interpreted as approximations of the exact dual potentials $\lambda^s$ in the prescribed finite-dimensional basis $\mathcal{W_\G}$.

We now approximate the velocity field. By \Cref{prop:3.2}, in the quadratic Benamou--Brenier setting, the optimal velocity field is given at the grid times by $v(x,t_s)=\frac{1}{2}\nabla\lambda^s(x)$. Therefore, from the moment dual potentials, we define the approximate velocity field associated with the fixed dual solution by
\begin{align}\label{eq:approximate v}
    \widehat{v}(x,t_s)=\frac{1}{2}\nabla\widehat{\lambda}^s(x), \quad s=0,\ldots,T-1.
\end{align}
The SOS identity \eqref{eq:SOS identity} gives a concrete formula for $\widehat{\lambda}^s$. Equating the terms in \eqref{eq:SOS identity} that depend only on $x$, we obtain, up to an additive constant,
\begin{align}\label{eq:approximate lambda}
    \widehat{\lambda}^s(x)=-\frac{\|x\|^2}{\Delta t_s}+R_x(\langle S^s,\Phi(z)\Phi(z)^*\rangle). 
\end{align}
Combining \eqref{eq:approximate v} and \eqref{eq:approximate lambda} then yields
\begin{align}
    \widehat{v}(x,t_s)=-\frac{x}{\Delta t_s} + \frac{1}{2}\nabla R_x(\langle S^s,\Phi(z)\Phi(z)^*\rangle). 
\end{align} 

In summary, the primal solution of the moment relaxation recovers low-order statistics of the intermediate distributions, while its dual potentials approximate the Benamou--Brenier potentials through an SOS relaxation. Differentiating these potentials gives the approximate Benamou--Brenier velocity field, which can then be used to generate the associated approximate flow and hence the corresponding Markov transition kernel in problem \eqref{problem:1.1}.

\section{Recovering general dynamics from problem \eqref{problem:coupling}}\label{sec:Ising}
The previous section shows that, when $\Omega$ is convex, the one-step cost is given by the kinetic cost \eqref{eq:dot-cost 3.1}, and no Markov kernel constraints are imposed (i.e., $\cD^s=\mathcal{P}(\Omega\times\Omega)$ in \eqref{constraint-control coupling}), problem \eqref{problem:coupling} exactly recovers the Benamou--Brenier geodesic. In more general settings, such as finite or nonconvex state spaces, the optimal solution of \eqref{problem:coupling} need not correspond to a classical Benamou--Brenier geodesic or to prescribed dynamics. Moreover, the convex relaxations in \Cref{sec:conv relaxation} retain only local statistics of the optimal sequential couplings, rather than the full one-step transition kernels. In this section, we introduce a parametric family of transition kernels $\{K_{\theta_s}^s\}_{\theta_s}$ and fit the parameters $\theta_s$ so that the one-step evolution induced by $K_{\theta_s}^s$ matches the recovered local statistics at the next time step.

We take the marginal relaxation as a representative case. After solving the relaxation, we extract the local marginals of the intermediate distributions
\begin{align*}
    \{\rho_A^s\}_{A\in\C}, \quad \{\rho_{AB}^s\}_{AB\in\E}, \quad s=0,\ldots,T.
\end{align*}
In the kernel-fitting step, we compare the next-time cluster-pair marginals predicted by the parameterized kernel with the recovered ones. The fitting loss may be evaluated on all retained cluster pairs, or only on those local marginals that are affected by the update and are informative for identifying the local transition law. Let $\E^s_{\fit}\subseteq \E$ denote this selected set of retained cluster pairs. For a cluster pair $AB \in \E_\fit^s$, define the next-time cluster-pair marginal predicted by the parameterized kernel as
\begin{align}\label{eq:next-time marginal}
    \widehat{\rho}_{AB,\theta_s}^{\, s+1}=({\rm P}_{AB})_\sharp \, \paren{(K_{\theta_s}^s)^*\rho^s},
\end{align}
where ${\rm P}_{AB}$ denotes the projection onto the coordinates in $AB$. Formula \eqref{eq:next-time marginal} is written at the full-law level. It can be evaluated locally whenever the $AB$-marginal of the one-step kernel depends on the current configuration only through a coordinate
neighborhood $N(AB)$. In a finite state space, this gives
\begin{align}
    \widehat{\rho}_{AB,\theta_s}^{\, s+1}(y_{AB})=\sum_{x_{N(AB)}} \rho_{N(AB)}^{\, s} \paren{x_{N(AB)}} K_{\theta_s,AB}^s \paren{y_{AB}\mid x_{N(AB)}},
\end{align}
where $K^s_{\theta_s,AB}$ denotes the output law on cluster $AB$ induced by $K^s_{\theta_s}$, and $\rho^s_{N(AB)}$ is the current marginal on the input coordinates needed to evaluate this local transition law. It remains to obtain the local input marginal $\rho^s_{N(AB)}$ from the relaxed solution. If $N(AB)$ is contained in a cluster or a retained cluster pair, then $\rho^s_{N(AB)}$ is obtained by projection from the corresponding relaxed marginal. Otherwise, we reconstruct it from the available one-cluster and cluster-pair marginals. A standard choice is the Bethe-type reconstruction. For example, on a path graph with single-coordinate clusters, suppose that an update at an interior site $i$ requires the current marginal on $\{i-1,i,i+1\}$, but only the pair marginals $\rho^s_{i-1,i}$ and $\rho^s_{i,i+1}$ are retained. Then the
Bethe-type reconstruction is
\begin{align*}
    \rho_{i-1,i,i+1}^{s}(a,b,c)=\frac{\rho_{i-1,i}^s(a,b)\rho_{i,i+1}^s(b,c)}{\rho_i^s(b)}, \quad a\in \Omega_{i-1}, \quad b \in \Omega_i, \quad c \in \Omega_{i+1}, \quad \rho_i^s(b) > 0.
\end{align*}

We then fit $\theta_s$ at each step by minimizing the squared loss $L^s(\theta_s)$:
\begin{align}
    L^s(\theta_s) \coloneqq \sum_{AB\in\E^s_{\fit}}\left\|\widehat{\rho}_{AB,\theta_s}^{\,s+1}-\rho_{AB}^{s+1}\right\|_2^2+\lambda R(\theta_s),
\end{align}
where $R$ is a regularization term and $\lambda\ge 0$ is the penalty strength. Thus the fitting step projects the recovered local statistics, in a least-squares sense, onto the chosen parametric Markov-kernel family.

In what follows, we illustrate this procedure for a constrained Markov process between two Ising models, where only one spin is allowed to change at each step. \Cref{subsec:4.1} formulates the Ising laws and imposes the single-spin update constraint in the sequential-coupling problem. \Cref{subsec:4.2} then uses the local marginals of the intermediate distributions obtained from the marginal relaxation to fit a time-inhomogeneous Glauber dynamics, thereby converting the recovered statistics into an explicit Markov kernel.

\subsection{Ising models and the single-spin update process}\label{subsec:4.1}

We consider the finite-state spin configuration space $\Omega = \set{-1,1}^d$, where each coordinate $x_i \in \set{-1,1}$ represents the spin at site $i$. An Ising model is specified by an interaction graph
\begin{align}
    \G_{\ising}=([d],\E_{\ising}), \quad \E_{\ising} \subset [d]_2 \coloneqq \set{\{i,j\} : 1 \le i < j \le d},
\end{align}
pairwise interaction strengths $\{J_{ij}\}_{ij\in \E_{\ising}}$, external fields $h=(h_1,\ldots,h_d)$, and inverse temperature $\beta>0$. For brevity, we write $ij$ for the unordered edge $\set{i,j}$. The corresponding probability mass function is
\begin{align}\label{eq:Ising model}
    p_{\theta}(x)=\frac{1}{Z_{\theta}}\exp\left[\beta\left(\sum_{ij\in \E_{\ising}}J_{ij}x_i x_j+\sum_{i=1}^d h_i x_i\right)\right], \qquad x\in\{-1,1\}^d,
\end{align}
where $\theta = (\G_{\ising},J,h,\beta)$ and $Z_{\theta}$ is the normalizing constant. The source and target laws are taken to be $\mu=p_{\theta_{\mu}}$ and $\nu=p_{\theta_{\nu}}$. The initial and terminal laws may have different interaction strengths, external fields, inverse temperatures, or interaction graphs. The graph $\G_{\ising}$ records the physical interaction structure of the Ising law and should be distinguished from the cluster graph $\G = (\C, \E)$ used in the relaxation scheme.

We impose a single-spin update constraint on the Markov process. Let $0 = t_0 < t_1< \cdots < t_T = 1$ be a time grid and let
\begin{align*}
    i_s\in[d], \qquad s=0,\ldots,T-1,
\end{align*}
be the active spin at time step $s$, during which only the spin $i_s$ is allowed to change while all other spins remain fixed. Thus the admissible Markov kernels satisfy
\begin{align}\label{eq:kernel constraint Ising}
    K^s(y\mid x)=0 \quad \text{whenever } (x,y)\notin \Gamma^s, \quad \Gamma^s \coloneqq \{(x,y)\in\Omega\times\Omega:y_j=x_j \text{ for all } j\ne i_s\}.
\end{align}
In the sequential-coupling formulation, this becomes the linear support constraint
\begin{align}
    \pi^s(x,y)=0 \quad \text{whenever } (x,y)\notin \Gamma^s.
\end{align}
The single-spin update process between the initial and terminal laws is therefore the following instance of the sequential-coupling problem \eqref{problem:coupling}:
\begin{align}
    \inf_{\{\pi^s\}_{s=0}^{T-1}}\; & \sum_{s=0}^{T-1}\sum_{x,y\in\Omega}c^s(x,y)\pi^s(x,y) \label{problem:Ising} \\
    \st  & \PL \pi^0=\mu,\quad \PR \pi^{T-1}=\nu, \label{constraint-endpoint Ising}\tag{4.8-a}\\
    & \PR \pi^s = \PL \pi^{s+1},\qquad s=0,\ldots,T-2, \label{constraint-time consis Ising}\tag{4.8-b}\\
    & \pi^s(x,y)=0 \quad \text{if } (x,y)\notin \Gamma^s, \label{constraint-control Ising}\tag{4.8-c}\\
    & \pi^s(x,y)\ge 0. \label{constraint-positivity Ising} \tag{4.8-d}
\end{align}
A simple one-step cost is the quadratic spin-flip cost $c^s(x,y)=\|x-y\|^2$. Under the single-spin update constraint, this cost is zero if the active spin does not change and equals $4$ if the active spin flips. Other local costs can also be used.

We solve \eqref{problem:Ising} using the marginal relaxation \eqref{problem:Mar1} from \Cref{subsec:marginal relaxation}. Since $\Omega$ is finite and the single-spin support constraints are affine, $\mathrm{Mar}^1$ becomes a linear program (LP). Let $\C=\{A_1,\ldots,A_K\}$ be a partition of $[d]$, and let $\G=(\C, \E)$ be the cluster graph used in the relaxation. Solving the resulting LP gives local marginals of the intermediate couplings. We then extract the local marginals of the intermediate distributions:
\begin{align}\label{eq:cluster marginals Ising}
    \rho_A^s, \quad \rho_{AB}^s, \quad A \in \C, \quad AB \in \E, \quad s = 0, \ldots, T.
\end{align}

\subsection{Fitting a Glauber dynamics}\label{subsec:4.2}

The marginal relaxation recovers local cluster marginals of the intermediate distributions as in \eqref{eq:cluster marginals Ising}, but it does not directly specify a Markov kernel. We now fit such a kernel from the recovered marginals using the method introduced at the beginning of \Cref{sec:Ising}. Because the admissible process has a single-spin update structure, we naturally choose a Glauber kernel. Specifically, over each time interval $[t_s, t_{s+1}]$, the Markov transition from $\rho^s$ to $\rho^{s+1}$ is parameterized as a single Glauber update with parameters $\theta_s$.

Glauber dynamics defines a classical Markov chain for sampling from Ising models. In our parameterization, the inverse temperature is absorbed into the effective external fields and local interactions, so a Glauber kernel is specified by a graph $\G_{\glb}$, external fields $h$, and interaction strengths $J$. The graph $\G_{\glb} = ([d], \E_{\glb})$ denotes the dependency graph for the fitted Glauber dynamics and is fixed across all time steps. This graph need not coincide with either the initial or the terminal Ising interaction graph, and it is distinct from the cluster graph $\G = (\C, \E)$ used in the marginal relaxation. Let $N_s$ be the neighborhood of the active spin $i_s$ at time $s$ induced by $\G_{\glb}$:
\begin{align*}
    N_s \coloneqq \set{j \in [d]: \{i_s,j\} \in \E_\glb}.
\end{align*}
At each update step of the Glauber dynamics, the active spin is resampled from its local conditional distribution while all other spins remain fixed:
\begin{align}\label{eq:Glauber conditional}
    q_{\theta_s}^s(x_{i_s} = \sigma\mid x_{N_s})=\frac{\exp(\sigma H^s(x_{N_s}))}{2\cosh(H^s(x_{N_s}))}, \qquad \sigma\in\{-1,1\}
\end{align}
where the local field associated with the neighboring configuration $x_{N_s}$ is defined as
\begin{align}\label{def:local field H^s}
    H^s(x_{N_s})=h_{i_s}^s+\sum_{j\in N_s}J_{i_sj}^s x_j.
\end{align}
Consequently, the one-step transition kernel is given by
\begin{align}\label{Markov kernel Ising}
    K_{\theta_s}^s(y\mid x)=\mathbf{1}_{\{y_j=x_j,\ j\ne i_s\}}q_{\theta_s}^s(y_{i_s}\mid x_{N_s}), \quad x, \;y \in \Omega.
\end{align}
The parameters to be fitted at step $s$ are $\theta_s = (h_{i_s}^s,\{J_{i_sj}^s:j\in N_s\})$. 

Next, we describe how this kernel predicts the cluster-pair marginals at the subsequent time step. Let $AB\in\E$ be a retained cluster pair. If the active spin $i_s \notin AB$, then all coordinates in $AB$ remain fixed during step $s$, yielding
\begin{align}
    \widehat{\rho}_{AB}^{\,s+1}(y_{AB})=\rho_{AB}^s(y_{AB}), \quad i_s\notin AB.
\end{align}
Such a marginal does not depend on $\theta_s$. If instead $i_s\in AB$, then the update probability of $Y_{i_s}$ depends on the neighboring spins in $N_s$. The predicted next-time cluster-pair marginal is computed from the current marginal on $(AB\setminus\{i_s\})\cup N_s$:
\begin{align}
    \widehat{\rho}_{AB, \theta_s}^{\,s+1}(y_{AB})=\sum_{x_{N_s\setminus AB}}\rho_{(AB\setminus\{i_s\})\cup N_s}^s (y_{AB\setminus\{i_s\}},x_{N_s\setminus AB}) \; q_{\theta_s}^s(y_{i_s}\mid x_{N_s}).
\end{align}
In the conditional law $q_{\theta_s}^s(y_{i_s}\mid x_{N_s})$, the vector $x_{N_s}$ is formed by taking $x_j = y_j$ for $j \in N_s \cap AB$, and summing over $x_j$ for $j \in N_s \setminus AB$. The marginal $\rho_{(AB\setminus\{i_s\})\cup N_s}^{\,s}$ is extracted from the recovered local marginals: if it is directly available from a retained cluster or cluster-pair marginal, we obtain it by projection; otherwise, it is reconstructed from the available local marginals by the Bethe-type local reconstruction. In particular, if the graph for Glauber dynamics $\G_\glb$ is chosen so that $N_s \subset AB$, then the expression simplifies to
\begin{align*}
    \widehat{\rho}_{AB, \theta_s}^{\,s+1}(y_{AB})=\rho_{AB\setminus\{i_s\}}^s(y_{AB\setminus\{i_s\}}) \; q_{\theta_s}^s(y_{i_s}\mid y_{N_s}).
\end{align*}
In summary, the predicted cluster-pair marginals are obtained by keeping the inactive coordinates fixed and resampling the active coordinate according to the fitted Glauber conditional law.

To estimate $\theta_s$, we match the predicted cluster-pair marginals with the recovered marginals at the subsequent time step. Let $\E^s_{\fit} = \{AB\in\E:i_s\in AB\}$ denote the set of retained cluster pairs used for fitting at step $s$. We solve
\begin{align}
    \min_{\theta_s}\sum_{AB\in\E^s_{\fit}}\left\|\widehat{\rho}_{AB, \theta_s}^{\,s+1}-\rho_{AB}^{s+1}\right\|_2^2+\lambda R(\theta_s),
\end{align}
where $\lambda\ge 0$ and $R(\theta_s)$ is a suitable regularization term.

The Glauber structure also admits a computationally efficient reduced form of this fitting procedure. Let $B_s=N_s\cup\{i_s\}$. If the recovered marginals are exactly generated by the single-site Glauber update, the log-odds ratio exhibits a linear structure:
\begin{align*}
    \frac{\rho_{B_s}^{s+1}(x_{N_s}=a,\;x_{i_s} = 1)}{\rho_{B_s}^{s+1}(x_{N_s}=a,\; x_{i_s} = -1)} = \exp(2H^s(a)),\quad a \in \set{-1,1}^{|N_s|}.
\end{align*}
Taking logarithms and using the definition of $H^s$ in \eqref{def:local field H^s} yields
\begin{align}\label{eq:Glauber fitting linear reg}
    \frac{1}{2}\log\frac{\rho_{B_s}^{s+1}(x_{N_s}=a, \; x_{i_s} = 1)}{\rho_{B_s}^{s+1}(x_{N_s}=a, \; x_{i_s} = -1)} = h_{i_s}^s+\sum_{j\in N_s}J_{i_sj}^s \, a_j.
\end{align}
Consequently, estimating $(h^s_{i_s},\{J^s_{i_sj}:j\in N_s\})$ reduces to a linear regression problem over configurations $a\in\{-1,1\}^{|N_s|}$. In practice, configurations with very small mass are either omitted or regularized by adding a small positive floor before taking the logarithm.

Repeating this procedure for $s = 0,\ldots, T-1$ yields a sequence of time-inhomogeneous Glauber kernels $\set{K_{\theta_s}^s}_{s=0}^{T-1}$ as defined in \eqref{Markov kernel Ising}. Given $X^0\sim\mu$, these kernels define a discrete-time Markov chain by
\begin{align*}
    X^{s+1}\sim K_{\theta_s}^s(\, \plh \mid X^s), \quad s=0,\ldots,T-1.
\end{align*}
The resulting process can be used to propagate samples in practice and approximately realize the recovered local intermediate statistics as marginals of an implementable Markov process.

\section{Numerical experiments}\label{sec:experiments}

In this section, we present numerical experiments illustrating the proposed relaxations and reconstruction methods. In \Cref{subsec:5.1}, we first test our method on Benamou--Brenier dynamics between Gaussian distributions, where the exact geodesic and velocity field are known. This example allows us to evaluate the accuracy of the recovered intermediate statistics and the extracted velocity field, as well as the effect of the cluster graph. We also compare our method with a particle-based back-propagation method to assess their empirical accuracy. In \Cref{subsec:5.2}, we consider a more challenging transport problem from a Gaussian distribution to a Ginzburg--Landau distribution. This experiment illustrates the reconstruction of intermediate low-dimensional marginals and the velocity-induced particle flow beyond the Gaussian setting. Finally, in \Cref{subsec:5.3}, we study a constrained Markov process between Ising distributions, where the marginal relaxation is combined with the fitting procedure in \Cref{sec:Ising} to fit Glauber dynamics.

\subsection{Benamou--Brenier dynamics between Gaussian distributions}\label{subsec:5.1}
We first consider the Benamou--Brenier dynamics between two Gaussian distributions $\mu = \N(m_0, \Sigma_0)$ and $\nu = \N(m_1, \Sigma_1)$. In this setting, both the Benamou--Brenier geodesic and the velocity field are available in closed form, so the recovered intermediate moments and velocity coefficients can be compared directly with the ground truth. 

The Benamou--Brenier geodesic is the Gaussian curve $\N(m_t, \Sigma_t)$, where
\begin{align}\label{eq:Gaussian mean cov}
    m_t=(1-t)m_0+tm_1, \quad \Sigma_t=B_t\Sigma_0B_t^\top,
\end{align}
with 
\begin{align*}
    B_t=(1-t)I+tQ, \quad Q = \Sigma_0^{-1/2}\left(\Sigma_0^{1/2}\Sigma_1\Sigma_0^{1/2}\right)^{1/2}\Sigma_0^{-1/2}.
\end{align*}
The optimal velocity field along the Benamou--Brenier geodesic is affine in $x$. More precisely, 
\begin{align}\label{eq:Gaussian velocity field}
    v(x,t)=A_tx+b_t, \quad A_t=(Q-I)B_t^{-1}, \quad b_t=m_1-m_0-A_tm_t.
\end{align}
These formulas are used to compute the reference mean, covariance, and optimal velocity field at the discrete time grid.

We sample the initial and terminal means $m_0, m_1 \in \R^d$ independently from the standard normal distribution $\N (0, I_d)$. The covariance matrices are generated through sparse precision matrices. More precisely, we first generate two symmetric banded matrices $P_0, P_1$ with bandwidth $3$. For $|i-j|\le 3$, $i\ne j$, the off-diagonal entries are sampled randomly, and the diagonal entries are chosen to make the precision matrices strictly diagonally dominant:
\begin{align*}
    (P_{\ell})_{ii}=\delta+\sum_{j\ne i}|(P_{\ell})_{ij}|,\qquad \ell=0,1,
\end{align*}
where $\delta>0$ is a diagonal shift. We then set $\Sigma_0 = P_0^{-1}$ and $\Sigma_1 = P_1^{-1}$. This construction gives Gaussian distributions with sparse precision matrices and decaying correlations.

We use single-coordinate clusters $\C = \set{\set{1}, \ldots, \set{d}}$. Let $G_{\mpath}$ be the one-dimensional path graph on $[d]$. For an integer radius $r \ge 1$, define the graph power $G^r_\mpath$ by connecting $i$ and $j$ whenever $\operatorname{dist}_{G_\mpath}(i,j)\le r$. In the experiments below, the cluster graph $\mathcal{G}$ is chosen as $G_\mpath^r$. We use the simplified moment relaxation \eqref{problem:Mom2} to solve the problem. For each coordinate $i$, we use the degree-one monomial basis $\set{\phi_{i,0}(x_i) = 1, \phi_{i,1}(x_i) = x_i}$. Thus each moment matrix $M^s = \pi^s\paren{\Phi \Phi^*}$ contains the first and second moments of two adjacent time layers, together with their cross-moments. The one-step cost $c^s(x,y) = \|x-y\|^2/\Delta t$ can be represented exactly by the selected basis. 

We set $d=50$, the precision-matrix diagonal shift $\delta=1$, the graph radius $r=16$, and $T=10$. We use an equispaced time grid on $[0,1]$ with $\Delta t = 1/T$. The resulting problem is a multi-block SDP with one dense semidefinite block $M^s$ per time interval, and is solved by MOSEK \cite{aps2019mosek} with a tolerance of $10^{-10}$.  Throughout this subsection, vector and matrix errors are reported as relative Euclidean and relative Frobenius errors, respectively. For sparse cluster graphs, we report a zero-filled covariance error. Specifically, for a recovered covariance matrix $\widehat{\Sigma}^s$ at time $t_s$, entries not retained by the cluster graph are set to zero: 
\begin{align*}
    \widehat{\Sigma}^s_{G_\mpath^r} = \widehat{\Sigma}^s\circ \mathbf{1}_{I\cup E(G_\mpath^r)}, \quad s = 0, \ldots, T
\end{align*}
Here, $I$ denotes the diagonal entries, $E(G^r_\mpath)$ denotes the retained off-diagonal entries, $\circ$ denotes entrywise multiplication, and $\mathbf{1}_{I\cup E(G^r_\mpath)}$ is the corresponding mask. The reported error is 
\begin{align*}
    \frac{\|\widehat{\Sigma}^s_{G_\mpath^r}-\Sigma_{t_s}\|_\mF}{\|\Sigma_{t_s}\|_\mF}, \quad s = 0, \ldots, T,
\end{align*}
where the ground truth $\Sigma_{t_s}$ is computed from \eqref{eq:Gaussian mean cov}. This metric captures both the recovery error on retained entries and the truncation error from omitted covariance entries. When $r=d-1$, the graph $G_\mpath^r$ is complete and the full covariance matrix is recovered from the relaxation.
\begin{figure}[H] 
    \centering  
    \captionsetup{labelformat=empty}
    \includegraphics[width=1\textwidth]{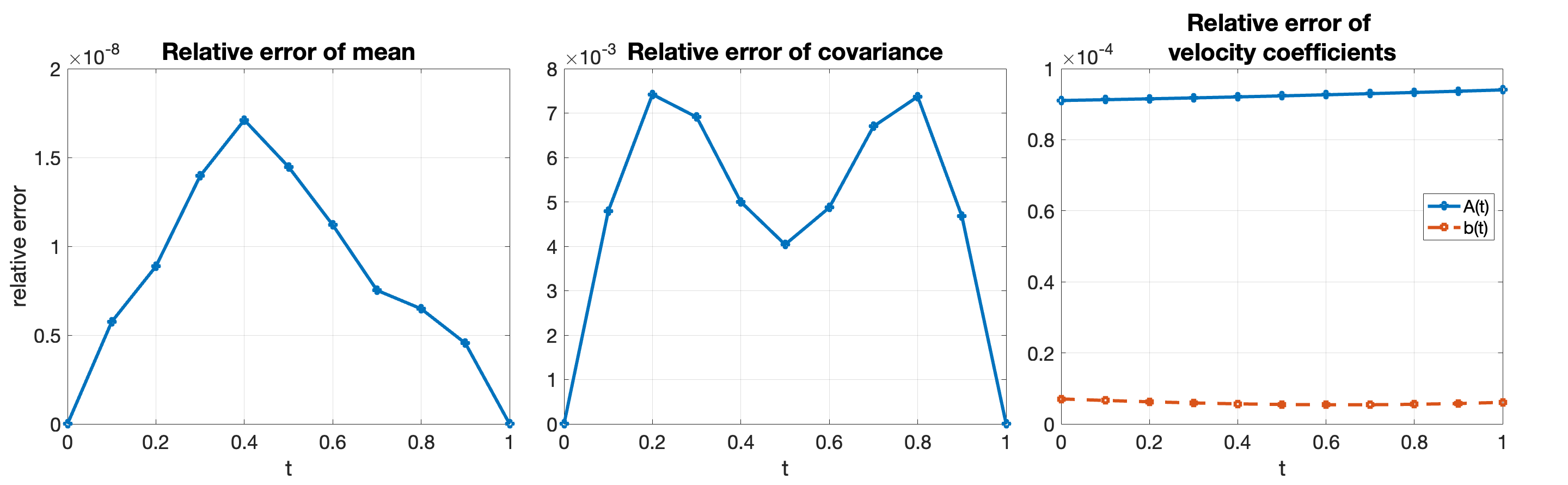} 
    \caption{Figure 1. Benamou–Brenier dynamics between Gaussian distributions in dimension $d=50$ with single-coordinate clusters and cluster graph $G_{\mpath}^{16}$. We use the moment relaxation \eqref{problem:Mom2}. The panels show relative errors of the recovered intermediate mean (left), covariance (middle), and velocity coefficients $A(t)$ and $b(t)$ (right) for the velocity field $v(x, t) = A(t)x + b(t)$.}
    \label{fig1} 
\end{figure}
\noindent\Cref{fig1} shows that the moment relaxation accurately recovers the intermediate statistics at the prescribed grid times and the affine velocity field. 

Next, we study the effect of the graph radius $r$. Keeping the source and target distributions unchanged from the previous setup, we vary $r$ from $6$ to $20$ in increments of $2$. For each $r$, we solve the corresponding relaxation associated with the cluster graph $G^r_\mpath$ and evaluate the covariance and velocity coefficient errors at $t=0.2$ and $t=0.5$.
\begin{figure}[H] 
    \centering  
    \captionsetup{labelformat=empty}
    \includegraphics[width=1\textwidth]{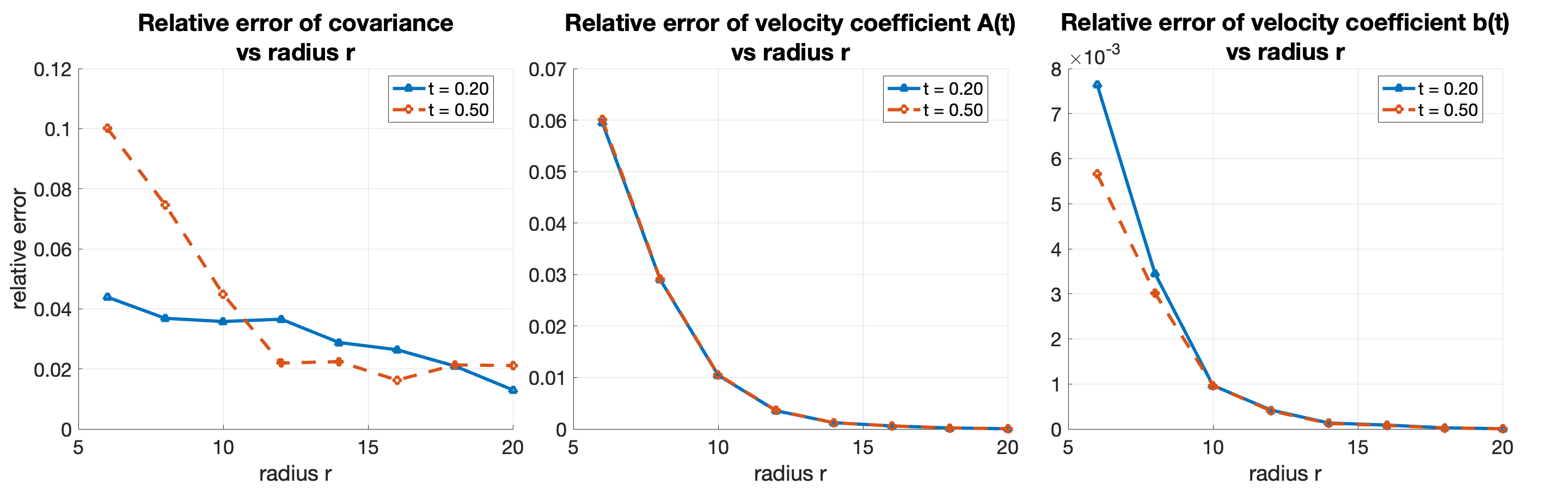} 
    \caption{Figure 2. Effect of the cluster graph radius on the recovered covariance and velocity field. The setting is the same as in Figure 1, except that the graph radius varies from $6$ to $20$. Errors are evaluated at $t \in \{0.2, 0.5\}$. The panels show relative errors for the recovered covariance (left), the affine velocity matrix $A(t)$ (middle), and the velocity offset $b(t)$ (right).}
    \label{fig2} 
\end{figure}
\noindent\Cref{fig2} shows that increasing the graph radius generally improves the recovered covariance and velocity coefficients. This behavior is consistent with the correlation decay of the Gaussian distribution: a larger graph radius retains more of the relevant covariance and cross-covariance structure.

Finally, we compare our approach with a particle-based back-propagation method. In this experiment, we choose dimension $d = 15$. The initial and terminal Gaussian distributions are generated as above with precision bandwidth $3$ and diagonal shift $\delta = 5$. The source and target means are sampled independently as
\begin{align*}
    m_0 \sim \N(0,I_d), \quad m_1 \sim \N(0.5 \, \mathbf{1}_d, I_d).
\end{align*}
To ensure a fair comparison, rather than using the closed-form Gaussian moments from previous experiments, we estimate the initial and terminal moments for the moment relaxation using $N_{\text{emp}} = 10{,}000$ independent samples drawn from each Gaussian distribution. The relaxation is then solved over the complete graph $G_\mpath^{d-1}$.

For the back-propagation method, we draw $N_{\text{train}}=10,000$ source particles $\set{X_i^0}_i \sim \mu$. We use the affine form of the Gaussian velocity field and parameterize the velocity field by
\begin{align*}
    v_\theta(x,t) = A_\theta(t)x + c_\theta(t).
\end{align*}
The time-dependent coefficients are expanded in the degree-two monomial basis $\set{1, t, t^2}$:
\begin{align*}
    A_\theta(t)=A_0+tA_1+t^2A_2, \quad c_\theta(t)=c_0+tc_1+t^2c_2,
\end{align*}
where $A_k \in \R^{d\times d}$ and $c_k \in \R^d$ are trainable parameters. The particles are pushed forward by explicit Euler steps
\begin{align*}
    X_i^{s+1} = X_i^s + \Delta t\, \bigl(A_\theta(t_s)X_i^s+c_\theta(t_s)\bigr).
\end{align*}
The empirical kinetic cost is
\[
    \mathcal{K}(\theta) = \sum_{s=0}^{T-1} \Delta t\, \frac{1}{N_{\text{train}}} \sum_{i=1}^{N_{\text{train}}} \left\| A_\theta(t_s)X_i^s+c_\theta(t_s)\right\|^2 .
\]
The terminal constraint at $t=1$ is enforced by matching Hermite moments of degree up to two. Let $\Psi(x)$ denote the Hermite basis vector centered and normalized with respect to the estimated target Gaussian distribution. The terminal residual is 
\[
    r(\theta) = \frac{1}{N_{\text{train}}} \sum_{i=1}^{N_{\text{train}}} \Psi(X_i^T) - \mathbb{E}_{Y\sim\nu}\Psi(Y),
\]
where the target expectation is estimated by Monte Carlo sampling from the target Gaussian distribution. While various penalties for the terminal constraint can be added to the kinetic cost, the moment-matching penalty presented here yields the best empirical performance. To further refine the solution, we employ an augmented Lagrangian method. The Lagrangian takes the form:
\[
    \mathcal{L}_\beta(\theta,\lambda) = \mathcal{K}(\theta) + \lambda^\top r(\theta) + \frac{\beta}{2} \|r(\theta)\|_2^2 .
\]
In the reported run, the penalty parameter is initialized as $\beta=5$ and kept fixed. The method performs $10$ outer augmented-Lagrangian iterations; in each outer iteration, the velocity parameters are optimized for $600$ Adam steps with learning rate $5\times 10^{-3}$. We use a gradient-clipping threshold of $100$, and a stepwise learning-rate scheduler with decay factor $0.6$ every $200$ inner iterations. After each inner solve, the dual coefficients are updated by $\lambda \leftarrow \lambda+\beta \, r(\theta)$. After training, we push forward a newly generated set of $N_{\text{test}} = 10,000$ particles using the learned velocity field, and compare the empirical covariance matrices along the path with the exact Gaussian geodesic covariance. We then compare these errors with the covariance error obtained directly from the moment relaxation. 
\begin{figure}[hbtp] 
    \centering  
    \captionsetup{labelformat=empty}
    \includegraphics[width=0.6\textwidth]{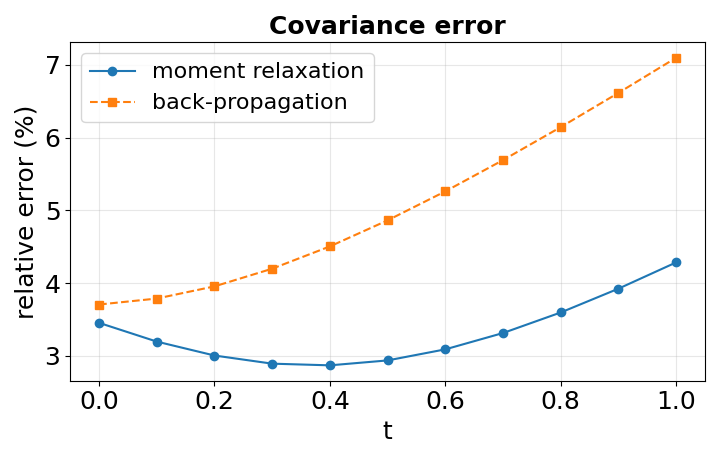} 
    \caption{Figure 3. Comparison between moment relaxation and back-propagation for the Gaussian benchmark in dimension $d=15$. The initial and terminal distributions are represented empirically using $10{,}000$ particles. The solid blue line indicates the relative error of the intermediate covariance using moment relaxation \eqref{problem:Mom2} with a complete cluster graph. For the dashed red line (back-propagation), an affine, time-dependent velocity field is trained on the $10{,}000$ particles via an augmented Lagrangian method. Intermediate covariance matrices are then computed empirically by pushing forward a new set of $10{,}000$ test particles along the trained velocity field.}
    \label{fig3} 
\end{figure}
\Cref{fig3} shows that the moment relaxation achieves a uniformly smaller covariance error along the trajectory in this experiment. The resulting SDP is also solved within a few seconds, substantially faster than the particle back-propagation training procedure.

\subsection{Benamou--Brenier dynamics between more general distributions}\label{subsec:5.2}

We next consider the Benamou--Brenier dynamics from a Gaussian distribution to a one-dimensional lattice Ginzburg--Landau distribution:
\begin{align}
    \nu(y) \propto \exp\left[-\beta\left(\sum_{i=1}^{d+1}\frac{\lambda}{2}\left(\frac{y_i-y_{i-1}}{h}\right)^2+\frac{1}{4\lambda}(1-y_i^2)^2\right)\right],\quad y\in[-L,L]^d, \quad h = \frac{1}{d+1}
\end{align}
with boundary conditions $y_0 = y_{d+1}=0$. The first term is a nearest-neighbor interaction, while the second is a double-well potential. Thus, the target distribution has local spatial correlation.

We set $d=10$, $\beta=1/8$, $\lambda=0.03$, $T=5$, and $L=2.5$. We use an equispaced time grid on $[0,1]$ with $\Delta t=1/T$. We generate $N_{\text{emp}}=50{,}000$ samples from $\nu$ by tensor-train conditional sampling \cite{dolgov2020approximation, peng2025}. Let $\widehat{m}_\nu$ and $\widehat{\Sigma}_\nu$ be the empirical mean and covariance of the training samples from $\nu$. We choose the source distribution to be a Gaussian distribution
\[
    \mu = \N(\widehat{m}_\nu, 0.7 \,\widehat{\Sigma}_\nu).
\]
Thus, the source and target have comparable locations and covariance scales. We draw another $N_{\text{emp}}=50{,}000$ samples from $\mu$ to estimate the initial moments and solve the Benamou--Brenier path from $\mu$ to $\nu$ using the moment relaxation \eqref{problem:Mom2}. We use the single-coordinate clusters $\C = \set{\set{1}, \ldots, \set{d}}$ and a path graph $G_\mpath$ on $[d]$ as the cluster graph $\G$. To recover moments for a non-neighboring cluster pair, such as the pair involving $x_1$ and $x_5$, we additionally include the edge $\{1,5\}$ in $\E$. 

We use a monomial cluster basis. For each coordinate, we take $\phi_{i,j}(x_i)=x_i^j$. Unlike the Gaussian case, where the path is fully determined by the first and second moments, the Ginzburg--Landau target requires higher-order moments to represent its non-Gaussian structure. For a cluster $A = \set{i}$, the monomial cluster basis is
\begin{align*}
    \Phi_A^{\mathrm{mon}}(z_A) = \set{x_i^{\alpha_x}y_i^{\alpha_y}: \alpha_x + \alpha_y \le \deg_{\mathrm{mon}}}.
\end{align*}
In the experiments below, we use $\deg_{\mathrm{mon}} = 11$. 

Using the above basis and parameters, we solve \eqref{problem:Mom2}. The problem is a multi-block SDP, with one real symmetric PSD block for each coupling $\pi^s$. To exploit sparsity, we apply chordal decomposition directly to each time-block PSD constraint. The resulting SDPs are solved by MOSEK with a tolerance of $10^{-10}$. After solving the chordally decomposed SDP, the clique-level primal matrices are reassembled into the original moment matrices by averaging overlapping clique entries and applying PSD completion. These completed matrices are then used to extract moments of intermediate distributions. The resulting objective values and implementation statistics are summarized in \Cref{tab:gl-sdp-results}. The static monomial method is the static cluster moment relaxation for optimal transport from \cite{khoo2025} and is included as a reference baseline for the endpoint transport problem. The dynamic monomial row corresponds to the moment relaxation \eqref{problem:Mom2} applied to problem \eqref{problem:dot}.

\begin{table}[t]
    \centering
    \caption{SDP results for the Ginzburg--Landau experiment}
    \label{tab:gl-sdp-results}
    \scriptsize
    \setlength{\tabcolsep}{3pt}
    \resizebox{\textwidth}{!}{
    \begin{tabular}{|c|c|c|c|c|c|c|}
        \hline
        Method 
        & Basis 
        & PSD size 
        & Chordal blocks
        & Constraints
        & SDP value
        & Time(s)
        \\
        \hline
        static monomial
        & $\deg_{\rm mon}=11$
        & $771$
        & $11$ $(\max=188)$
        & $31370/34430$
        & $0.5428$
        & $99.55$
        \\
        \hline
        dynamic monomial
        & $\deg_{\rm mon}=11$
        & $771\times 5$
        & $55$ $(\max=188)$
        & $148835/166430$
        & $0.5331$
        & $539.09$
        \\
        \hline
    \end{tabular}}
    \vspace{0.5em}
    
    \begin{minipage}{0.96\textwidth}
    \footnotesize
    The static row reports the optimal transport relaxation from \cite{khoo2025} as a reference, while the dynamic row reports the moment relaxation applied to the sequential-coupling problem. The column \emph{Constraints} reports the number of affine constraints before and after chordal decomposition. The \emph{PSD size} is the original semidefinite block size before chordal decomposition. The \emph{SDP value} provides a lower bound for the squared Wasserstein distance.
    \end{minipage}
\end{table}

We reconstruct the intermediate low-dimensional marginals from the recovered moments by a maximum-entropy step. Let $M^s_{AB} (\rho^s)$ denote the recovered moments of the $AB$-marginal of the distribution $\rho^s$. These moments are extracted from the left- or right-layer moments of the coupling $\pi^s$. At the continuous level, the maximum-entropy reconstruction solves
\begin{align}
    \max_{\rho^s_{AB}} \quad & -\int  \log \frac{\dd \rho_{AB}^s }{\dd x_{AB}} \; \dd \rho^s_{AB} \tag{ME}\\
    \st & \rho^s_{AB}(R_x(\Phi_A \Phi_B^*)) = M^s_{AB} (\rho^s), \nonumber \tag{ME-a} \label{constraint-moment ME}\\
    & \rho_{AB}^s([-L,L]^{|AB|}) = 1.\tag{ME-b}
\end{align}
Here $\rho^s_{AB}$ is optimized over all nonnegative measures on $[-L,L]^{|AB|}$. In the computation, we solve a discretized version on a grid and relax the moment-matching constraint \eqref{constraint-moment ME} by
\begin{align*}
    \|\rho^s_{AB}(R_x(\Phi_A \Phi_B^*)) - M^s_{AB} (\rho^s)\|_{\infty}\le \varepsilon_{\mathrm{mom}}\max\{1,\|M^s_{AB} (\rho^s)\|_{\infty}\}.
\end{align*}
We use $\varepsilon_{\mathrm{mom}}=10^{-2}$ and $21$ grid points per coordinate. In \Cref{fig4}, the resulting $2$-dimensional marginals are visualized by kernel density estimation on $[-L,L]^2$ with $120$ grid points per coordinate. Here we do not use all moments available from the degree-$11$ SDP relaxation. Although the relaxation contains monomial moments up to order $22$, high-degree monomial moments are poorly scaled and can dominate the moment-matching constraints. Therefore we only use recovered monomial moments with total degree $|\alpha| \le 4$ in the maximum-entropy reconstruction.
\begin{figure}[H]  
    \centering  
    \captionsetup{labelformat=empty}
    \begin{subfigure}{\textwidth}
        \centering
        \includegraphics[width=1\textwidth]{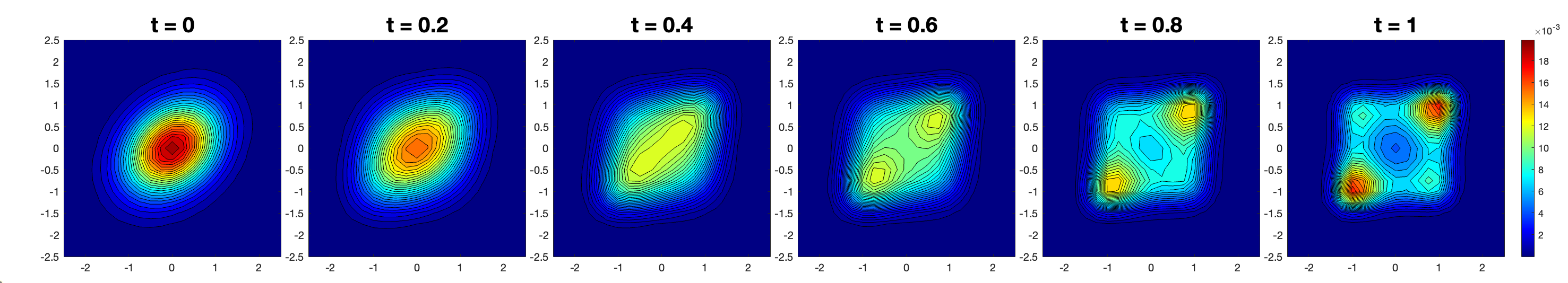} 
    \end{subfigure}
    \vspace{0.5em} 
    \begin{subfigure}{\textwidth}
        \centering
        \includegraphics[width=1\textwidth]{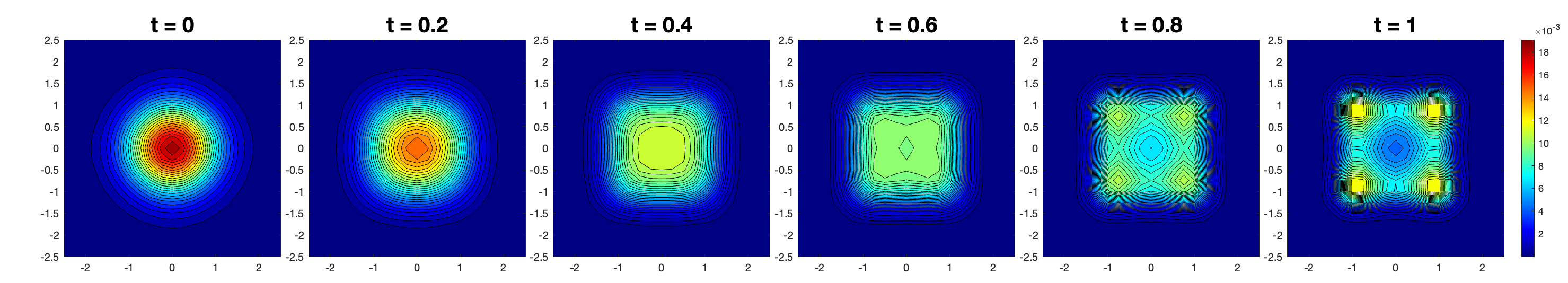} 
    \end{subfigure}
    \caption{Figure 4. Intermediate two-dimensional marginals for transport from a Gaussian law to a Ginzburg--Landau law in dimension $d=10$. Maximum-entropy reconstructions, based on moments obtained via moment relaxation \eqref{problem:Mom2}, show a gradual deformation toward the double-well target. The first row displays the $(1,2)$-marginal, capturing local interaction effects, while the second row presents the $(1,5)$-marginal, illustrating longer-range correlations.}
    \label{fig4}
\end{figure}

\Cref{fig4} shows that the maximum-entropy reconstructions from the recovered moments capture the gradual deformation from the Gaussian source to the Ginzburg--Landau target. The neighboring $(1,2)$-marginal captures local correlations induced by the interaction term, while the non-neighboring $(1,5)$-marginal illustrates how the relaxation can recover selected longer-range low-dimensional statistics when the corresponding cluster-pair edge is retained.

Finally, we use the dual slack matrices to compute the velocity field $\set{\widehat{v}(\plh, t_s)}_{s=0}^{T-1}$ as discussed in \Cref{subsec:3.3}. We then push forward test samples from $\mu$ by explicit Euler steps. On each interval $[t_s,t_{s+1}]$, we use five Euler substeps:
\begin{align*}
    X^{s,q+1}=X^{s,q}+\frac{\Delta t}{5}\widehat v(X^{s,q}, t_s),\qquad q=0,\ldots,4.
\end{align*}
The final value $X^{s,5}$ is used as the initial value for the next time interval. To stabilize the visualization, we only transport test samples satisfying $\|X\|_{\infty}\le 2.2$ and clip the transported samples to $[-L,L]^d$. 
\begin{figure}[H]  
    \centering  
    \captionsetup{labelformat=empty}
    \begin{subfigure}{\textwidth}
        \centering
        \includegraphics[width=0.9\textwidth]{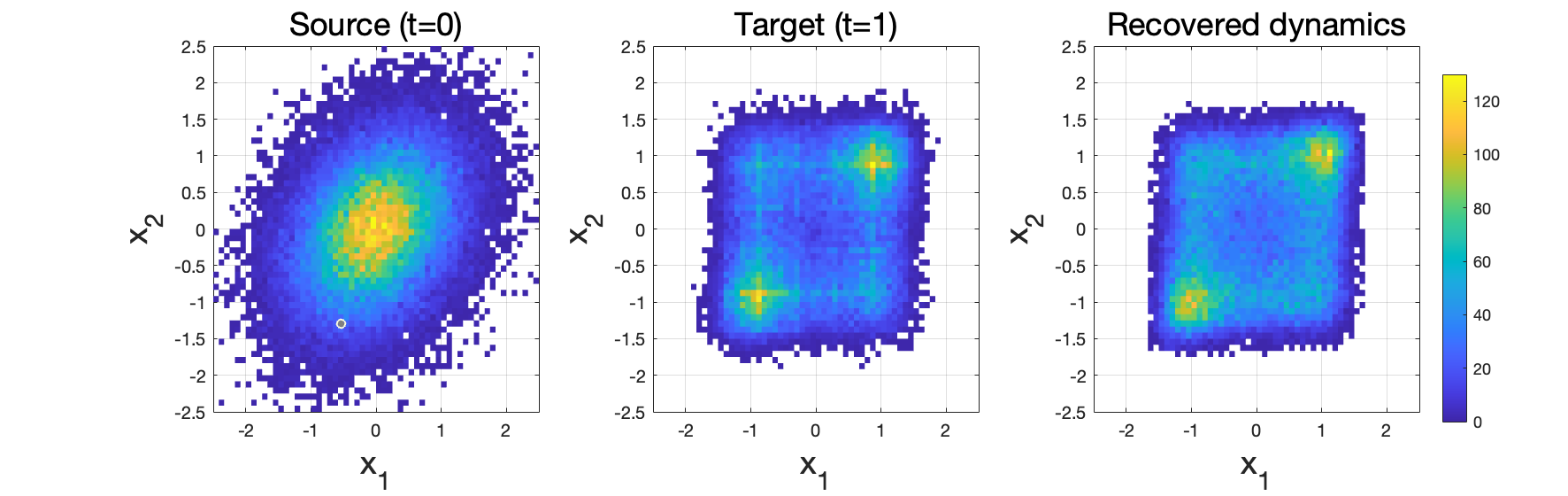} 
    \end{subfigure}
    \vspace{0.5em} 
    \begin{subfigure}{\textwidth}
        \centering
        \includegraphics[width=0.9\textwidth]{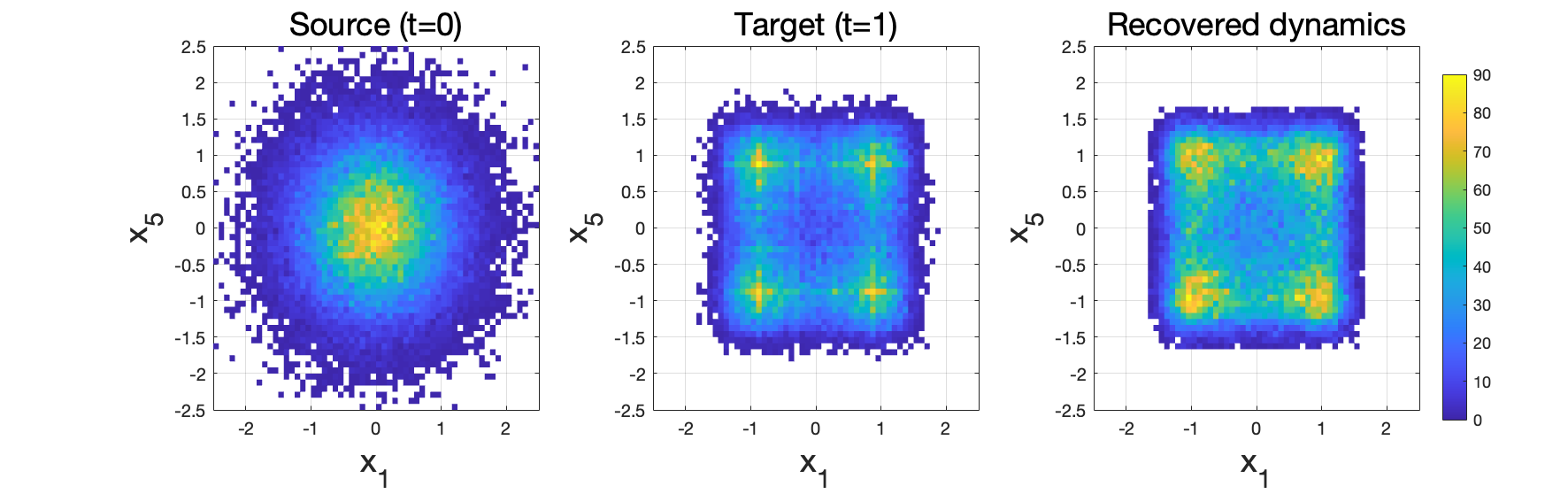} 
    \end{subfigure}
    \caption{Figure 5. Velocity-induced push-forward for transport from a Gaussian law to a Ginzburg--Landau law in dimension $d=10$. Panels compare the empirical source (left) and target (middle) distributions with the source pushed forward to $t=1$ (right) via the velocity field recovered from the computed dual solution of \eqref{problem:Mom2}. Rows $1$ and $2$ display the empirical $(1,2)$- and $(1,5)$-marginals, respectively.}
    \label{fig5}
\end{figure}
\noindent In \Cref{fig5}, we compare the empirical $(1,2)$- and $(1,5)$-marginals across three cases: the source samples, the target samples, and the push-forward samples generated by the recovered velocity fields at $t = 1$. This demonstrates that the computed dual solution of the moment relaxation provides an implementable approximate flow, not merely intermediate moment statistics.

\subsection{Single-spin update process between Ising models}\label{subsec:5.3}

We finally consider the constrained Markov process between Ising models introduced in \Cref{sec:Ising}. The Ising model is specified in \eqref{eq:Ising model}. We write $(\beta_\mu, h_\mu, J_\mu)$ and $(\beta_\nu, h_\nu, J_\nu)$ for the parameters of the source law $\mu$ and target law $\nu$, respectively. In the experiments below, the initial and terminal laws share the same inverse temperature $\beta_\mu = \beta_\nu = 0.4$, but have opposite uniform edge interactions $J_{\mu,ij} = 1$ and $J_{\nu,ij} = -1$ on the chosen Ising graph. Thus $\mu$ is ferromagnetic and favors aligned neighboring spins, whereas $\nu$ is antiferromagnetic and favors alternating neighboring spins. The external fields are set to zero: $h_\mu = h_\nu = 0$. We generate $N_{\text{emp}}=20{,}000$ samples $\set{X_r}_{r=1}^{N_{\text{emp}}} \sim \mu$ and $\set{Y_r}_{r=1}^{N_{\text{emp}}} \sim \nu$ using Glauber dynamics, and impose the initial and terminal constraints using the empirical distributions.

We consider two types of Ising graphs $\G_\ising$. The first is the one-dimensional path graph on  $[d]$, with edge set $\E_{\ising} = \set{\{i,i+1\}:i = 1, \ldots, d-1}$. For this $1$D Ising model, we choose $d=30$ and cluster size $n_c=2$. We decompose the coordinate set into contiguous clusters
\[
    A_k=\{(k-1)n_c+1,\ldots,\min(k\,n_c,d)\}, \quad k = 1, \ldots, \left\lceil{d/n_c}\right\rceil.
\]
The cluster graph is chosen to be the path graph on these contiguous clusters.

The second example is a $2$D Ising model on a $d_x\times d_y$ grid. In the experiment below, we take $d_x=d_y=4$, so the total number of spins is $d=16$. We partition the lattice into four $2 \times 2 $ clusters and choose the cluster graph to be a cycle, as shown in \Cref{fig6}.
\begin{figure}[H]
\centering
\begin{tikzpicture}[
    scale=0.75,
    every node/.style={font=\small},
    spin/.style={circle, draw, thick, minimum size=7mm, inner sep=0pt},
    clusterbox/.style={draw, dashed, rounded corners, thick, inner sep=6pt},
    cnode/.style={font=\small}
]

\foreach \c in {0,1,2,3}{
    \foreach \r in {0,1,2,3}{
        \pgfmathtruncatemacro{\n}{4*\c+\r+1}
        \node[spin] (v\n) at (1.7*\c,-1.7*\r) {\n};
    }
}
\foreach \r in {0,1,2,3}{
    \foreach \c in {0,1,2}{
        \pgfmathtruncatemacro{\a}{4*\c+\r+1}
        \pgfmathtruncatemacro{\b}{4*(\c+1)+\r+1}
        \draw[thick] (v\a) -- (v\b);
    }
}
\foreach \c in {0,1,2,3}{
    \foreach \r in {0,1,2}{
        \pgfmathtruncatemacro{\a}{4*\c+\r+1}
        \pgfmathtruncatemacro{\b}{4*\c+\r+2}
        \draw[thick] (v\a) -- (v\b);
    }
}
\node[clusterbox, fit=(v1)(v2)(v5)(v6), label=above:$A_1$] {};
\node[clusterbox, fit=(v3)(v4)(v7)(v8), label=below:$A_2$] {};
\node[clusterbox, fit=(v9)(v10)(v13)(v14), label=above:$A_3$] {};
\node[clusterbox, fit=(v11)(v12)(v15)(v16), label=below:$A_4$] {};
\node[cnode] at (2.55,-7.5) {(a) 2D Ising graph with cluster decomposition};
\begin{scope}[xshift=10.5cm, yshift=-1.2cm]
    \node[spin] (A1) at (0,0) {$A_1$};
    \node[spin] (A3) at (2.8,0) {$A_3$};
    \node[spin] (A2) at (0,-2.8) {$A_2$};
    \node[spin] (A4) at (2.8,-2.8) {$A_4$};
    \draw[thick] (A1) -- (A3);
    \draw[thick] (A1) -- (A2);
    \draw[thick] (A2) -- (A4);
    \draw[thick] (A3) -- (A4);
    \node[cnode] at (1.4,-4.0) {(b) Cluster graph};
\end{scope}

\end{tikzpicture}
\caption{Cluster decomposition and cluster graph for the $4\times 4$ $2$D Ising model.}
\label{fig6}
\end{figure}
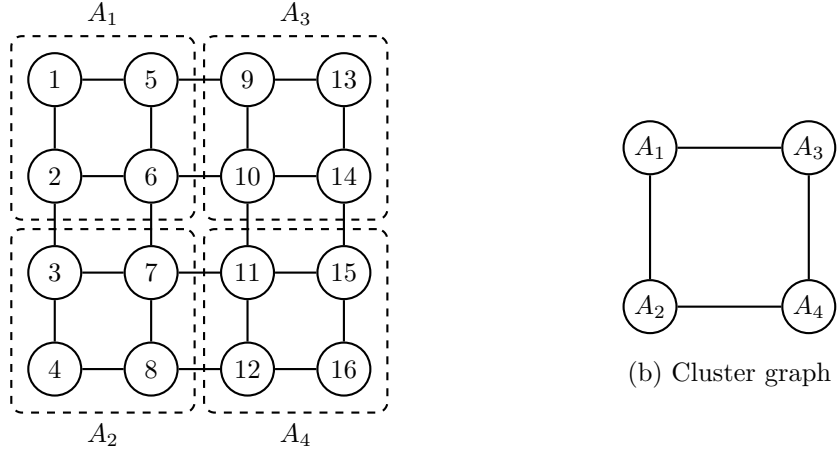

For both the $1$D and $2$D models, the source and target Ising parameters are chosen as above. We take the number of time steps $T=5d$, corresponding to five complete sweeps over the spin system. The active spin at step $s$ is chosen cyclically as 
\[
    i_s = 1 + (s \bmod d), \quad s=0, \ldots, T-1. 
\]
We use the quadratic spin-flip cost $c^s(x,y) = \|x-y\|^2$, which equals $4$ if the active spin flips and $0$ otherwise. We solve the problem using the marginal relaxation \eqref{problem:Mar1}, which becomes a linear program. In the implementation, we use a reduced version adapted to the single-spin schedule. At time step $s$, coupling variables are introduced only for clusters $A$ with $i_s \in A$ and retained cluster-pairs $AB$ with $i_s \in AB$. Marginals of inactive blocks are propagated from the most recent time at which the corresponding block was updated. This reduced construction gives the same relevant local marginals for the single-spin schedule while substantially reducing the LP size. Solving the LP yields the recovered marginals, denoted by $\{\rho_A^s\}_{A\in\C}$ and $\{\rho_{AB}^s\}_{AB\in\E}$ for $s=0,\ldots,T$.

For the $1$D case, after solving the relaxation, we track the two-spin marginal on sites $\{1,2\}$ after each epoch. Here, a single epoch corresponds to a complete sweep of all spins. \Cref{fig7} shows the evolution of the $(1,2)$-marginal. The intermediate heatmaps display a gradual transfer of probability mass from aligned configurations to anti-aligned configurations.  
\begin{figure}[H] 
    \centering  
    \captionsetup{labelformat=empty}
    \includegraphics[width=1.0\textwidth]{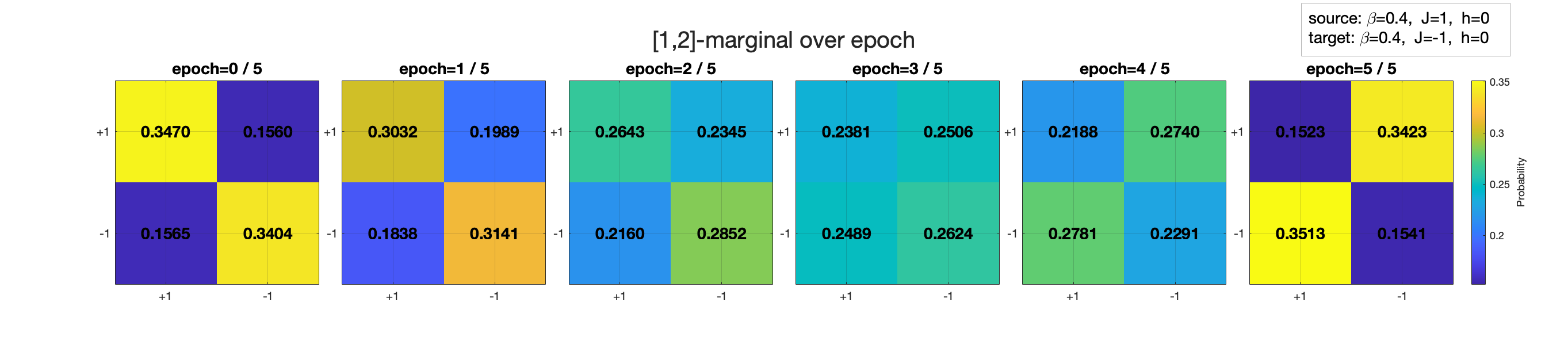} 
    \caption{Figure 7. Recovered marginals on cluster $A_1 = \{1,2\}$ for the 1D Ising experiment with $d=30$ spins. Across the five epochs, the recovered marginal transfers mass from aligned spin configurations, favored by the ferromagnetic source, to anti-aligned configurations, favored by the antiferromagnetic target.}
    \label{fig7} 
\end{figure}

After obtaining the local marginals from the relaxation, we fit Glauber dynamics using the linear regression form \eqref{eq:Glauber fitting linear reg}. For the $1$D experiment, we take the Glauber dependency graph to be the path graph on $[d]$. For an interior active spin $i=i_s$, the parameters $(h_i^s, J_{i,i-1}^s, J_{i,i+1}^s)$ are fitted from the recovered next-time local marginal on $B_i = \set{i-1,i,i+1}$ by weighted ridge regression. Since $B_i$ is contained in a retained cluster pair, its marginal can be obtained directly by projection from the relaxed solution. For boundary spins, the same procedure is used with $B_1 = \set{1,2}$ and $B_d = \set{d-1,d}$. The fitted parameters define a sequence of Glauber kernels $\{K_{\widehat{\theta}_s}^s\}_{s=0}^{T-1}$. Starting from the same source samples $\set{X^0_r}_{r=1}^{N_{\mathrm{emp}}}$, we generate particles by
\begin{align*}
    X^{s+1}_r\sim K_{\widehat{\theta}_s}^s(\cdot\mid X^s_r), \qquad s=0,\ldots,T-1.
\end{align*}
From the generated particles, we compute empirical local marginals and compare them with the marginals obtained from the relaxation.
\begin{figure}[H] 
    \centering  
    \captionsetup{labelformat=empty}
    \includegraphics[width=1.0\textwidth]{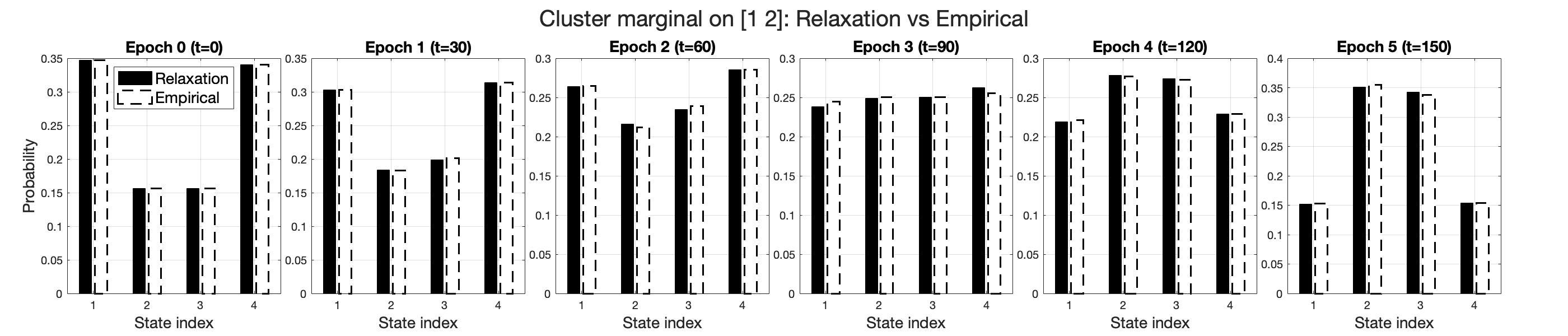} 
    \caption{Figure 8. Fitted Glauber dynamics reproduces the recovered two-spin marginals in the $1$D Ising experiment with $d=30$ spins. The solid black bars represent the marginals on $A_1 = \{1, 2\}$ recovered via the marginal relaxation \eqref{problem:Mar1}, while the open bars with dashed outline correspond to the empirical marginals generated by the fitted kernels. The state indices $1-4$ denote the spin configurations $(+1,+1)$, $(+1,-1)$, $(-1,+1)$, and $(-1,-1)$, respectively.}
    \label{fig8} 
\end{figure}
\noindent \Cref{fig8} compares the recovered marginals from the marginal relaxation and the empirical marginals of the fitted Glauber dynamics on the cluster $\set{1,2}$. The bars almost overlap throughout the process, showing that the fitted dynamics closely reproduces the recovered two-spin marginal on this block.

We also visualize a single particle trajectory generated by the fitted dynamics. \Cref{fig9} shows snapshots of one representative particle after each epoch. The trajectory changes by single-spin updates along the cyclic schedule, illustrating one realization of the fitted Markov process whose empirical marginals match the recovered marginals.
\begin{figure}[H] 
    \centering  
    \captionsetup{labelformat=empty}
    \includegraphics[width=1.0\textwidth]{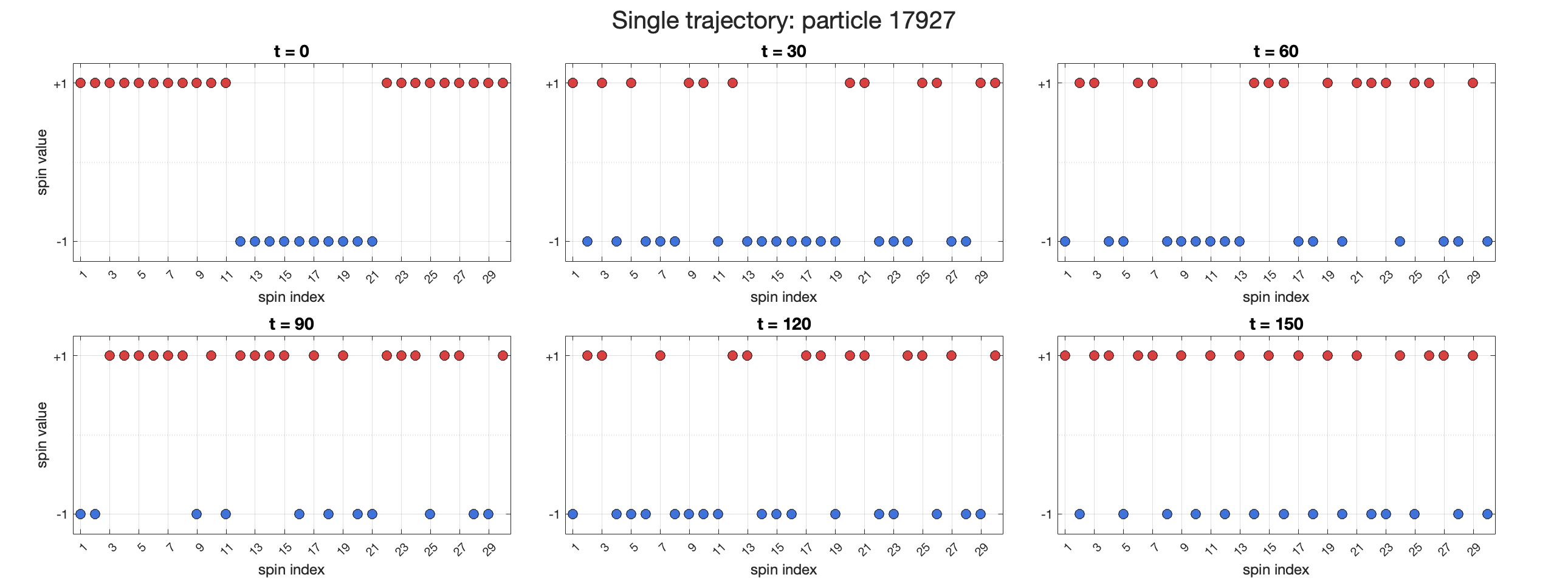} 
    \caption{Figure 9. Representative trajectory generated by the fitted Glauber dynamics in the 1D Ising experiment. The panels illustrate the time evolution of a single sample path, transitioning from a ferromagnetic configuration at $t=0$ to an antiferromagnetic configuration at $t=150$.}
    \label{fig9} 
\end{figure}

For the $2$D case, we use the same initial and terminal parameters and the $2 \times 2$ cluster structure described above. The marginal relaxation is solved using the same single-spin schedule. In the fitting step, we use a dense local graph for Glauber dynamics induced by the retained cluster-pair supports: if $AB \in \E$ is an edge in the cluster graph, then all spin pairs inside $AB$ are included in $\G_\glb$. The local marginal required for fitting the update of spin $i_s$ is not always directly available from a single retained cluster or cluster-pair marginal, so we approximate it by the Bethe-type reconstruction. \Cref{fig10} compares the recovered marginals from the marginal relaxation and the empirical marginals generated by the fitted Glauber dynamics on the cluster $A_1 = \set{1,2,5,6}$. The fitted process reproduces the recovered local marginal on $A_1$ well across the displayed epochs.
\begin{figure}[htbp] 
    \centering  
    \captionsetup{labelformat=empty}
    \includegraphics[width=1.0\textwidth]{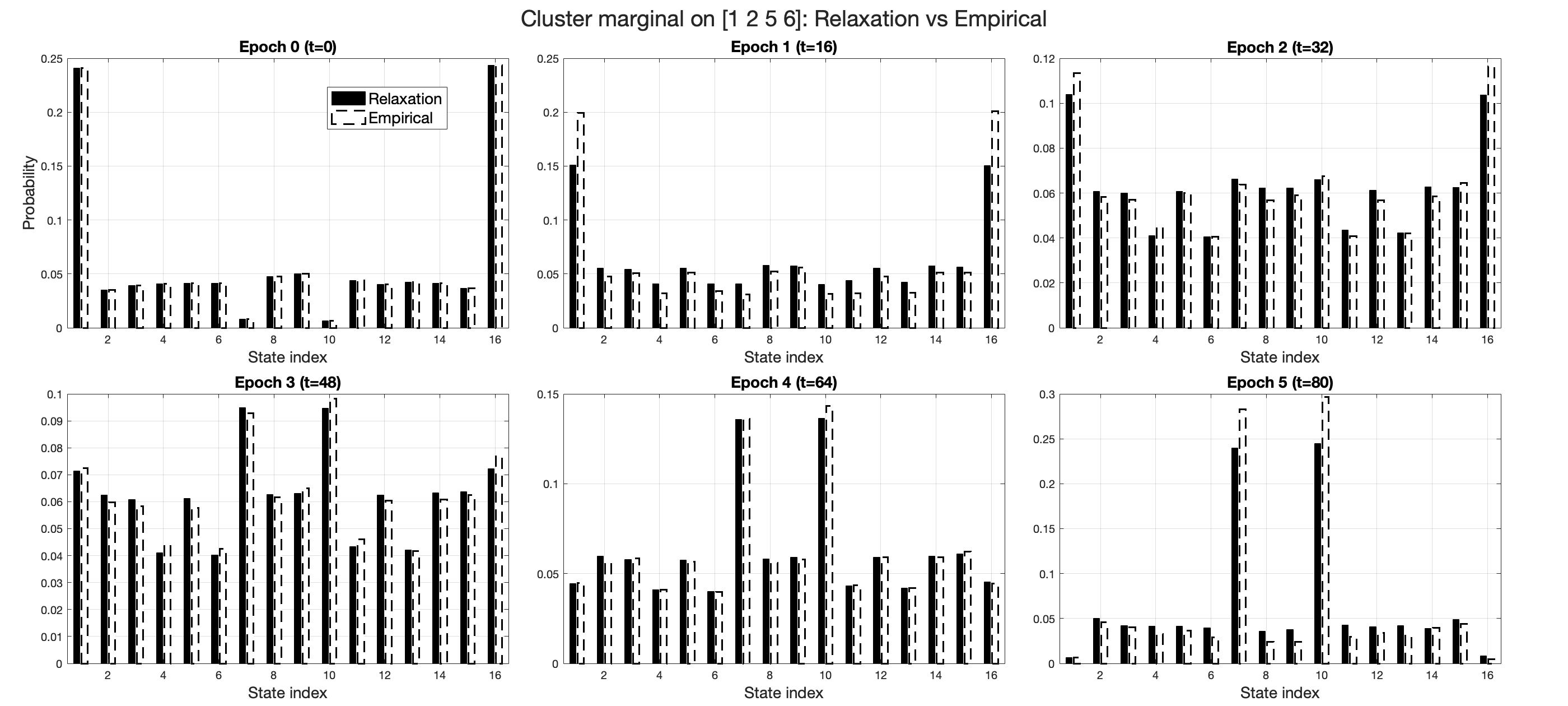} 
    \caption{Figure 10. Fitted Glauber dynamics reproduces the recovered four-spin marginals in the $2$D Ising experiment with $d=16$ spins. The solid black bars represent the marginals on $A_1 = \{1, 2,5,6\}$ recovered via the marginal relaxation \eqref{problem:Mar1}, while the open bars with dashed outline correspond to the empirical marginals generated by the fitted kernels. The $16$ states on the x-axis correspond to all possible configurations of $(X_1, X_2, X_5, X_6)$, sorted in standard binary order (mapping $+1 \to 0$, $-1 \to 1$) where $X_6$ alternates most rapidly.}
    \label{fig10} 
\end{figure}

\section{Conclusion and discussion}\label{sec:conclusion}

In this work, we extended the convex-relaxation framework developed for optimal transport in \cite{khoo2025} to the optimization of Markov processes via a time-discrete sequential-coupling formulation. The proposed relaxations avoid explicit representation of the full joint law while retaining local statistics of the intermediate distributions, thereby mitigating the curse of dimensionality. In the unconstrained kinetic-cost setting, the framework recovers the Benamou–Brenier geodesic on the prescribed time grid, and the dual variables provide a practical way to reconstruct approximate velocity fields. Beyond the Benamou–Brenier setting, we showed how the recovered local statistics can be converted into implementable dynamics by fitting a parametric transition family, as illustrated by Glauber dynamics for finite spin systems.

Several directions remain open. It would be important to characterize conditions under which the locality of the initial and terminal distributions is preserved along the interpolating process. Such a result would directly justify sparse marginal and cluster-moment relaxations at intermediate times. Another direction is to extend the framework from Markov processes with prescribed transition structures to broader classes of controlled dynamics. Rotation-constrained transport is one representative example: instead of moving mass through an unconstrained velocity field, one seeks to steer a distribution through structured transformations. Incorporating such control variables into the relaxation would broaden the scope of the convex framework for process optimization.

\newpage
\printbibliography

\end{document}